\documentclass[10pt,electronic]{kthesis}
\usepackage[T1]{fontenc}
\usepackage{lmodern,amsfonts,amsthm,amsmath,amssymb,verbatim,enumerate,eepic}
\usepackage{mathrsfs}
\usepackage[utf8]{inputenc}
\usepackage{epsfig}
\usepackage{pdfpages}
\usepackage{color}
\usepackage[swedish,french,german,english]{babel}
\usepackage{enumerate}
\usepackage{macros}
\usepackage{cite}

\usepackage{float}
\usepackage{algorithm, algorithmicx}
\usepackage[hidelinks]{hyperref}
\usepackage{multirow}
\usepackage{cite}

 \newcommand{\inv}{^{-1}}
 \newcommand{\T}{^T}

\theoremstyle{plain}
\newtheorem{theorem}{Theorem}[section]
\newtheorem{lemma}[theorem]{Lemma}
\newtheorem{proposition}[theorem]{Proposition}

\theoremstyle{definition}

\theoremstyle{remark}

\makepagestyle{minech}
\makeoddfoot{minech}{}{\thepage}{}
\makeevenfoot{minech}{}{\thepage}{}

\copypagestyle{mine}{ruled}
\makeoddhead{ruled}{}{}{Limited-memory quasi-Newton methods for a quadratic function}
\makeevenhead{ruled}{}{}{}
\makeevenfoot{mine}{}{\thepage}{}
\makeoddfoot{mine}{}{\thepage}{}

\counterwithout{section}{chapter}
\counterwithout{figure}{section}
\counterwithout{table}{section}
\setsecnumdepth{subsection}
 \makeatletter
  \renewcommand{\fnum@figure}{\textbf
  {\figurename~\thefigure}}
  \renewcommand{\fnum@table}{\textbf
  {\tablename~\thetable}}
 \makeatother    

\captionnamefont{\footnotesize}
\captiontitlefont{\footnotesize}

\begin{document}
\chapterstyle{article}
\chapter*[]{Exact linesearch limited-memory quasi-Newton methods for minimizing a
  quadratic function}
\thispagestyle{minech} % removes header

David Ek\footnote{Optimization and Systems Theory, Department of
  Mathematics, KTH Royal Institute of Technology, SE-100 44 Stockholm,
  Sweden (\texttt{daviek@kth.se,andersf@kth.se}).\label{fn} } and Anders Forsgren\footref{fn}   \footnote[0]{Research supported by the Swedish Research Council (VR).}

\begin{abstract}
  The main focus in this paper is exact linesearch methods for
  minimizing a quadratic function whose Hessian is positive
  definite. We give a class of limited-memory quasi-Newton Hessian
  approximations which generate search directions parallel to those of
  the method of preconditioned conjugate gradients, and hence give
  finite termination on quadratic optimization problems in exact
  arithmetic. With the framework of reduced-Hessians this class
  provides a dynamical framework for the construction of
  limited-memory quasi-Newton methods. We give an indication of the
  performance of the methods within this framework by showing
  numerical simulations on sequences of related systems of linear
  equations, which originate from the CUTEst test collection.
  
  In addition, we give a compact representation of the Hessian
  approximations in the full Broyden class for the general
  unconstrained optimization problem. This representation consists of
  explicit matrices and gradients only as vector components.
  
  \medskip\noindent
  \textbf{Keywords:} method of conjugate gradients, quasi-Newton method, unconstrained quadratic program, limited-memory method, exact linesearch method.
  
\end{abstract}

\section{Introduction} \label{intro}
In this work we mainly study the behavior of limited-memory quasi-Newton methods on unconstrained quadratic optimization problems on the form
\begin{equation}
\min_{x \in \mathbb{R}^n} \frac{1}{2} x^T H x + c^Tx, 
\label{eq:QP} \tag{QP}
\end{equation}
where $H=H^T$ and $H \succ 0$. (Throughout, ``$\succ$'' is used to
denote positive definite and we denote the corresponding set of matrices by $\mathcal{S}_+^n$, i.e. $\mathcal{S}_+^n = \{ B_0 \in \mathbb{R}^{n \times n} : B_0 = B_0^T \mbox{ and } B_0 \succ 0 \}$.) In particular, we study exact linesearch
limited-memory quasi-Newton methods that generate search directions
parallel to those of the method of preconditioned conjugate gradients
(PCG). Under exact linesearch parallel search
directions imply identical iterates. 

The motivation for this work originates from applications in which it is desired to solve or approximately solve a sequence of related systems of linear equations. In particular systems where the matrix is symmetric positive definite. Such sequences occur when solving unconstrained nonlinear optimization problems with Newton's method and in interior-point methods, which constitute some of the most widely used methods in numerical
optimization. As the problems become larger the arising systems of
linear equations typically become increasingly computationally
expensive to solve and iterative methods may be considered. In exact
arithmetic, our model method is the method of preconditioned conjugate
gradients, but this method may be too inaccurate in finite
precision. Quasi-Newton methods may be expected to be significantly
more accurate, but the computational cost is typically too high. In
consequence, we aim for less computationally expensive limited-memory
versions of quasi-Newton methods that are more accurate than the
method of preconditioned conjugate gradients. The goal is to provide
better understanding of whether it is viable and/or efficient to use
such methods to solve or approximately solve systems of linear equations
that arise as Newton's method or interior-point methods converge.

We envisage the use of the limited-memory quasi-Newton methods as an
accelerator for a direct solver when solving a sequence of systems of
linear equations. E.g., when the direct solver and the iterative
solver can be run in parallel, and where the preconditioner is updated
when the direct solver is faster for some system of linear equations
in the sequence.

We mainly propose a framework for the construction of reduced-Hessian
limited-memory quasi-Newton methods. To give an indication of their
potential we construct two examples in this framework for which
we give numerical results.

Limited-memory quasi-Newton
methods have previously been studied by various authors, e.g., as
memory-less quasi-Newton methods by Shanno \cite{Sha78n},
limited-memory BFGS (L-BFGS) by Nocedal \cite{Noc80n} and more
recently as limited-memory reduced-Hessian methods by Gill and Leonard
\cite{GilLeo03}. In contrast, we specialize to exact linesearch
methods for problems on the form (\ref{eq:QP}). The model method is
PCG, which is interpreted as a particular quasi-Newton method as is
done by e.g., Shanno \cite{Sha78n} and Forsgren and Odland
\cite{ForOdl18}. We start from a result by Forsgren and Odland
\cite{ForOdl18}, which provides necessary and sufficient conditions on
the Hessian approximation for exact linesearch methods on (\ref{eq:QP})
to generate search directions that are parallel to those of PCG. The
focus is henceforth directly on Hessian approximations with this
property. The approximations are described by a novel compact
representation which contains explicit matrices together with
gradients and search directions as vector components.  The framework
for the compact representation is first given for the full Broyden
class where we consider unconstrained optimization problems on the
form
\begin{equation} \tag{P}
\min_{x \in \mathbb{R}^n} f(x),
\label{eq:uncP}
\end{equation}
where the function $f: \mathbb{R}^n \rightarrow \mathbb{R}$ is assumed to be smooth. Compact representations of quasi-Newton matrices have previously been used by various authors but were first introduced by Byrd, Nocedal and Schnabel \cite{ByrNocSna94}. They were thereafter extended to the convex Broyden class by Erway and Marcia \cite{ErwMar14,ErwMar17}, and to the full Broyden class by DeGuchy, Erway and Marcia \cite{DeG17}. In contrast, we give an alternative compact representation of the Hessian approximations in the full Broyden class which only contains explicit matrices and gradients as vector components. In addition we discuss how exact linesearch is reflected in this representation. 

Compact representations of limited-memory Hessian approximations in
the Broyden class are also discussed by Byrd, Nocedal and Schnabel
\cite{ByrNocSna94} and Erway and Marcia \cite{ErwMar17}. In contrast,
our discussion is on limited-memory representations of Hessian
approximations intended for exact linesearch methods for problems on
the form (\ref{eq:QP}), and the approximations are not restricted to
the Broyden class. In addition, our alternative representation
provides a dynamical framework for the construction of limited-memory
approximations for the mentioned purpose.

In Section~\ref{section:BG} we provide a brief background to
quasi-Newton methods, unconstrained quadratic optimization problems
(\ref{eq:QP}) and to the groundwork that provides the basis for this
study. Section~\ref{section:CR} contains the alternative compact
representation for the full Broyden class. In Section~\ref{section:QP}
we present results which include a class of limited-memory Hessian
approximations together with a discussion of how to solve the
systems of linear equations that arise using reduced-Hessian methods. Section~\ref{section:numRes} contains numerical results on randomly generated problems on the form (\ref{eq:QP}) and on systems of linear equations which originate from the CUTEst test collection \cite{DolMor2002}. Finally in Section~\ref{section:conc} we give some concluding remarks. 

\subsection{Notation}
Throughout, $\mathcal{R}(M)$ and $\mathcal{N}(M)$ denote the range and
the nullspace of a matrix $M$ respectively. Moreover, $e_i$ denotes
the $i$th unit vector of the appropriate dimension and $| \mathcal{S} |$
denotes the cardinality of a set $\mathcal{S}$.

\section{Background} \label{section:BG}
In this section we give a short introduction to quasi-Newton methods for unconstrained optimization problems on the form (\ref{eq:uncP}). Thereafter, we give a background to unconstrained quadratic optimization problems (\ref{eq:QP}) and to the groundwork that provides the basis for this study.  
\subsection{Background on quasi-Newton methods}
Quasi-Newton methods were first introduced as variable metric methods
by Davidon~\cite{Dav91} and later formalized by Fletcher and Powell~\cite{FlePow63}. For a thorough introduction to quasi-Newton methods see, e.g., \cite[Chapter 3]{Fle87a} and \cite[Chapter 6]{NocWri06}. In quasi-Newton methods the search direction, $p_k$, at iteration $k$ is generated by
\begin{equation} 
B_k p_k = - g_k,
\label{eq:QN}
\end{equation}
where $B_k$ is an approximation of the true Hessian $\nabla^2f(x_k)$ and $g_k$ is the gradient $\nabla f(x_k)$. The symmetric two-parameter class of Huang \cite{Hua70} satisfies the scaled secant condition
\begin{align}
B_ks_{k-1} = \sigma_k y_{k-1},
\label{eq:scaledSecCond}
\end{align}
where $s_{k-1} = x_k - x_{k-1}$, $y_{k-1} = g_k - g_{k-1}$ and $\sigma_k$ is one of the free parameters. The most well-known quasi-Newton class is obtained if $\sigma_k=1$ in (\ref{eq:scaledSecCond}), namely the one-parameter Broyden class  which updates $B_{k-1}$ to
\begin{align}
B_k & =  B_{k-1} - \frac1{s_{k-1}^T B_{k-1} s_{k-1}} B_{k-1} s_{k-1} s_{k-1}^T B_{k-1}  +  \frac1{y_{k-1}^Ts_{k-1}} y_{k-1}y_{k-1}^T \nonumber \\ & \quad +\phi_{k-1} \omega_{k-1} \omega_{k-1}^T,
\label{eq:broydenFamys}
\end{align}
where
\[
\omega_{k-1} = \big(s_{k-1}^TB_{k-1}s_{k-1} \big)^{1/2}\Big(\frac1{y_{k-1}^Ts_{k-1}} y_{k-1} - \frac1{s_{k-1}^TB_{k-1}s_{k-1}} B_{k-1}s_{k-1} \Big),
\]
with $\phi_{k-1}$ as the free parameter \cite{Fle94}. The
Broyden-Fletcher-Goldfarb-Shanno (BFGS) update scheme is obtained if
$\phi_{k-1}=0$ and Davidon-Fletcher-Powell (DFP) if $\phi_{k-1}=
1$. In this work we study Hessian
approximations described by compact representations with gradients and search directions as vector
components. We will therefore throughout this work explicitly use the
quantities $g$, $p$ and the steplength $\alpha$ in all equations. In
this notation, the Broyden class Hessian approximations in (\ref{eq:broydenFamys}) may be written as
\begin{align}
B_k & =  B_{k-1} + \frac1{g_{k-1}^T p_{k-1}} g_{k-1} g_{k-1}^T \nonumber \\ &\quad +  \frac1{\alpha_{k-1}\left( g_k - g_{k-1} \right)^Tp_{k-1}} \left( g_k - g_{k-1} \right)\left( g_k - g_{k-1} \right)^T  +\phi_{k-1} \omega_{k-1} \omega_{k-1}^T,
\label{eq:broydenFam}
\end{align}
where
\[
\omega_{k-1} = \left(-g_{k-1}^Tp_{k-1} \right)^{1/2}\left( \frac1{\left( g_k - g_{k-1} \right)^Tp_{k-1}} \left( g_k - g_{k-1} \right) - \frac1{g_{k-1}^T p_{k-1}} g_{k-1} \right).
\]
As shown in (\ref{eq:broydenFam}), the previous
Hessian approximation is in general updated by a rank-two matrix with range equal to the space spanned by the current and the previous
gradient. Furthermore, it is well known that under exact linesearch
all Broyden class updates generates identical iterates, as shown by Dixon~\cite{Dix72_1}. 
The case $\phi_{k-1}=0$ in (\ref{eq:broydenFam}), i.e., the BFGS
update, will have a particular role in part of our analysis. We will
refer to quantities $B_k$, $p_k$ and $\alpha_k$ corresponding to this
case as $B_k^{BFGS}$, $p_k^{BFGS}$ and $\alpha_k^{BFGS}$.

\subsection{Background on quadratic problems}
Solving (\ref{eq:QP}) is equivalent to solving the linear system 
\begin{equation}
Hx+c = 0,
\label{eq:QPoptCond}
\end{equation}
which has a unique solution if $H \succ 0$. The problem (\ref{eq:QP}), and hence (\ref{eq:QPoptCond}), is in this work solved by an exact linesearch method on the following form. The steplength, iterate and gradient at iteration $k$ is updated as
\[
\alpha_{k} = -\frac{g_{k}^T p_{k}}{p_{k}^T H p_{k}}, \qquad x_{k+1} =  x_{k} + \alpha_{k} p_{k}, \qquad g_{k+1} = g_{k} + \alpha_{k} H p_{k},
\]
which together with a specific formula for $p_k$ constitute the particular exact linesearch method. The model method is summarized in the Algorithm~\ref{alg:ELSMQP} below.  
%\vspace{-5mm}
\begin{algorithm}[H]
\caption{An exact linesearch method for solving (\ref{eq:QP}).}
\begin{footnotesize}
\begin{algorithmic}[1]
\Statex $k \gets 0$, \qquad $x_k \gets$ Initial point, \qquad $g_k \gets Hx_k + c$ 
\Statex \textbf{While} $\| g_k \| \neq 0$ \textbf{do}
\Statex\hspace{\algorithmicindent} $p_k  \gets $ search direction 
\Statex\hspace{\algorithmicindent} $\alpha_{k}  \gets -\frac{g_{k}^T p_{k}}{p_{k}^T H p_{k}}$
\Statex\hspace{\algorithmicindent} $x_{k+1} \gets x_k + \alpha_k p_k$
\Statex\hspace{\algorithmicindent} $g_{k+1} \gets g_k + \alpha_k H p_k$ 
\Statex\hspace{\algorithmicindent} $k  \gets k+1$
\Statex \textbf{End}
\end{algorithmic}
\end{footnotesize}
\label{alg:ELSMQP}
\end{algorithm} 
%\vspace{-5mm}
%\vspace{-2mm}
\noindent The search direction in Algorithm~\ref{alg:ELSMQP} may be calculated using PCG with a symmetric positive definite preconditioner $M$. The corresponding algorithm for solving (\ref{eq:QPoptCond}) can be formulated using the Cholesky factor $L$ defined by $M = L L^T$. This is equivalent to application of the method of conjugate gradients (CG) to the preconditioned linear system
\begin{equation}
L^{-1}HL^{-T} \hat{x} + L^{-1}c = 0, 
\label{eq:PcondQPoptCond}
\end{equation}
with $\hat{x}= L^T x$, see, e.g., Saad \cite[Chapter 9.2]{Saa03}. If all quantities generated by CG on (\ref{eq:PcondQPoptCond}) are denoted by "$\>\hat{\>}\>$", then these quantities will relate to those from CG on (\ref{eq:QPoptCond}) as, $\hat{g} = L^{-1}g$ and $\hat{p} = L^{T}p$. The iteration space when $M = I$ or when $M$ is an arbitrary symmetric positive definite matrix will thus be related through a linear transformation. 
In this work the following PCG update is considered,
\begin{equation}
p_k^{PCG} =  
\begin{cases}
-M\inv g_0 & k = 0, \\
-M^{-1}g_{k} + \frac{g_k^TM^{-1}g_k}{g_{k-1}^TM^{-1} g_{k-1}} p_{k-1} & k \geq 1.
\end{cases}
\label{eq:pPCG}
\end{equation}
The discussion in this work is mainly on Hessian approximations $B_k$ that generate $p_k$ parallel to $p_k^{PCG}$.  We will therefore hereinafter only consider the preconditioner $M = B_0$ where $B_0 \in \mathcal{S}_+^n$. If no preconditioner is used, i.e. $B_0=I$, then (\ref{eq:pPCG}) is the update referred to as Fletcher-Reeves, which together with the exact linesearch method of Algorithm~\ref{alg:ELSMQP} is equivalent to the method of conjugate gradients by Hestenes and Stiefel \cite{HesSte52}. If the search direction
(\ref{eq:pPCG}) is used in Algorithm~\ref{alg:ELSMQP}, the method
terminates when $\| g_r \| = 0 $ for some $r$ where $r \leq n$ and $x_r$ solves
(\ref{eq:QP}). The search directions generated by the method are
mutually conjugate with respect to $H$ and satisfy $p_i \in
\textit{span}\big( \{B_0\inv g_0, \dots, B_0\inv g_i \} \big)$,
$i=0,\dots,r$. In addition, it holds that the generated gradients are
mutually conjugate with respect to $B_0\inv$, i.e. $g_i\T B_0\inv g_j=0$,
$i\ne j$. By expanding (\ref{eq:pPCG}), the search direction of PCG may be expressed as
\begin{equation}
p_k^{PCG} = -g_k^T B_0^{-1} g_k \sum_{i=0}^k \frac{1}{g_i^TB_0^{-1}g_i}B_0^{-1}g_i.
\label{eq:pPCGexp}
\end{equation}
Forsgren and Odland~\cite{ForOdl18} have provided necessary and
sufficient conditions on $B_k$ for an exact linesearch method to
generate $p_k$ parallel to $p_k^{PCG}$. This result provides
the basis of our work and is therefore reviewed. Under exact
linesearch and $B_0 \in \mathcal{S}_+^n$ it holds that $g_k^T
p_k^{PCG} = - g_k^T B_0 \inv g_k$. With $p_{k-1}^{PCG} =
p_{k-1}/\delta_{k-1}$, (\ref{eq:pPCG}) can be written as
\begin{equation} \label{eq:qNPCG}
p_{k}^{PCG} = - C_k \inv B_0 \inv g_k,
\end{equation}
where 
\begin{equation} \label{eq:Ck}
C\inv _k = I + \frac{1}{g_{k-1}^Tp_{k-1}} p_{k-1}g_k^T,
\text{which imply} \
C_k = I - \frac{1}{g_{k-1}^Tp_{k-1}} p_{k-1}g_k^T.
\end{equation}
Insertion of $p_k^{PCG} = p_k / \delta_k = -B_k \inv g_k / \delta_k$ into (\ref{eq:qNPCG}) gives 
\begin{equation} \label{eq:iffCondOnBk}
B_k C\inv _k B_0 \inv g_k = \frac{1}{\delta_k} g_k.%\left( 1 / \delta_k\right) g_k. %\frac{1}{\delta_k} g_k.
\end{equation}
Premultiplication by $C_k^{-T}$ while noting that $C_k^{-T} g_k= g_k$ and letting \linebreak$W_k = C_k^{-T} B_k C_k \inv$ yield 
\begin{equation} \label{eq:iffCondOnBk2}
B_k = C_k^T W_k C_k,  \text{\ with\ } W_k B\inv _0 g_k =  \left( 1 / \delta_k\right) g_k, %  \frac{1}{\delta_k} g_k,  
\end{equation}
for $W_k$ nonsingular. Finally, it holds that $B_k \succ 0$ if and only if $W_k \succ 0$. For further details, see \cite[Proposition 4]{ForOdl18}.

With the exact linesearch method of Algorithm~\ref{alg:ELSMQP} for
solving (\ref{eq:QP}), parallel search directions imply identical iterates,
and therefore search directions parallel to those of PCG imply finite
termination. Huang has shown that the quasi-Newton Huang class, the Broyden class and PCG generate parallel search directions \cite{Hua70}.

Finally we review a result which is related to the conjugacy of the search directions. Part of the result is similar to those given by Fletcher in \cite{Fle87a}. The result will have a central part the analysis to come. 

\begin{lemma} \label{lemma:conjProp} Consider iteration $k$, $1 \leq k < r$, of the exact linesearch method of Algorithm~$\ref{alg:ELSMQP}$ for solving $(\ref{eq:QP})$. For $B_0 \in \mathcal{S}_+^n$ and $p_i=\delta_i p_i^{PCG}$, $\delta_i \neq 0$, \linebreak $i = 0, \dots, k-1$, then $p_k = \delta_k p_k^{PCG}$, $\delta_k \neq 0$ if and only if 
\begin{equation}
g_i^T p_k = c_k \neq 0, \quad i=0,\dots, k,
\label{eq:conjProp}
\end{equation}
with $c_k = -\delta_k g_k ^T B_0\inv g_k$ and
$p_k\in\textit{span}\left(\{B_0\inv g_0, \dots, B_0\inv g_k \}\right)$.
\end{lemma}
\begin{proof}
  Note that by the assumptions, $g_i$, $i=0,\dots,k$, are identical to those
  generated by PCG. We first show the only-if direction. Premultiplication of $p_k^{PCG}$ in
  (\ref{eq:pPCGexp}) by $g_i^T$ while taking into
  account the conjugacy of the $g_j$'s with respect to $B_0\inv$ gives
  $g_i^T p_k^{PCG}=-g_k^T B_0\inv g_k$, so that $g_i^T (\delta_k
  p_k^{PCG})=c_k$ for $c_k = -\delta_k g_k^T B_0\inv g_k$. In addition,
  (\ref{eq:pPCGexp}) shows that $p_k \in \textit{span}\left( \{B_0\inv g_0,
  \dots, B_0\inv g_k \} \right)$.
To show the other direction, let 
\begin{equation}\label{eq:pIspan}
p_k=\sum_{j=0}^k \gamma_j B_0\inv g_j,
\end{equation}
Premultiplication of (\ref{eq:pIspan}) by $g_i^T$ while taking into
account the conjugacy of the $g_j$'s with respect to $B_0\inv$ gives $
g_i^T p_k = \gamma_i g_i^T B_0\inv g_i,$ $\quad i=0,\dots,k,$ hence if $p_k$ satisfies (\ref{eq:conjProp}) with $c_k = -\delta_k g_k^T B_0\inv g_k$, then it follows that
\begin{equation}\label{eq:gammak}
\gamma_i =-\delta_k  g_k^T B_0\inv g_k/g_i^T B_0\inv g_i, \quad i=0,\dots,k.  % \frac{ g_k^T B_0\inv g_k}{g_i^T B_0\inv g_i}, \quad i=0,\dots,k.
\end{equation}
Insertion of (\ref{eq:gammak}) into (\ref{eq:pIspan}) gives $p_k =
\delta_k p_k^{PCG}$, with $p_k^{PCG}$ given by (\ref{eq:pPCGexp}). 
\end{proof}
For further details on methods of conjugate gradients and related analyses, see e.g. \cite{Pyt2009}.

\section{A compact representation of Broyden class Hessian approximations}
\label{section:CR}
In this section we consider unconstrained optimization problems on the
form (\ref{eq:uncP}) and give a compact representation of the Hessian
approximations in the full Broyden class. The representation contains
only explicit matrices and gradients as vector components. In this
section, the gradient of $f$ in (\ref{eq:uncP}) will be denoted by $g$,
in contrast to all other sections where $g$ refers to the gradient of
the objective function of (\ref{eq:QP}). We first give the general
representation without exact linesearch and then discuss how exact
linesearch is reflected in the representation. 
\begin{lemma} \label{lemma:CRall} Consider iteration $k$ of solving
  $(\ref{eq:uncP})$ by a quasi-Newton method. Let $B_0$ be a given
  nonsingular matrix. Assume that $p_i$, $i=0,\dots, k-1$, has been
  given by $B_i p_i = -g_i$ with $B_i$, $i=1,\dots,k-1$, as any
  nonsingular matrix on the form $(\ref{eq:broydenFam})$. Any Hessian
  approximation in the Broyden class can then be written as
\begin{equation*}
B_k = B_0 +\sum_{i=0}^{k-1} \left[\frac1{g_{i}^T p_i}g_i g_i^T +
  \frac1{\alpha_i
    (g_{i+1}-g_i)^Tp_i}(g_{i+1}-g_i)(g_{i+1}-g_i)^T+\phi_i \omega_i
  \omega_i^T\right]
%\label{eq:broydenFamII}
\end{equation*}
where
\[
  \omega_i = \left(-g_i^T p_i\right)^{1/2}\left(\frac1{
    (g_{i+1}-g_i)^T p_i}(g_{i+1}-g_i) -
  \frac1{g_i^T p_i} g_i \right), 
\]
or equivalently
\begin{equation} \label{eq:cmpRep:outerProdForm}
B_k = B_{0} + G_k T_k G_k^T,
\end{equation}
where $G_k = \begin{bmatrix}
g_{0} & g_1 & \dots & g_{k-1} & g_k 
\end{bmatrix} \in \mathbb{R}^{n \times (k+1)}$ and $T_k \in \mathbb{R}^{(k+1) \times (k+1)}$ is a symmetric tridiagonal matrix on the form $T_k = T_k^{C} + \mathit T_k^{\phi}$, with
\begingroup\makeatletter\def\f@size{10}\check@mathfonts
\def\maketag@@@#1{\hbox{\m@th\normalfont#1}}%
\begin{subequations}
\begin{align}
e_1^T T_k^C e_1 &= \frac1{g_{0}^T p_0} + \frac1{\alpha_0
  (g_{1}-g_0)^Tp_0},\label{eq:Tcfirst} \\ 
e_{i+1}^T T_k^C e_{i+1}  & = \frac1{g_{i}^T p_i} + \frac1{\alpha_{i-1} (g_{i}-g_{i-1})^Tp_{i-1}} \nonumber \\ 
& \quad + \frac1{\alpha_i
  (g_{i+1}-g_i)^Tp_i}, \qquad \qquad  \qquad  \quad \qquad i=1,\dots,k-1, \\ 
  e_{i+1}^T T_k^C e_{i}  &= e_{i}^T T_k^C e_{i+1} = -\frac1{\alpha_{i-1}
  (g_{i}-g_{i-1})^Tp_{i-1}},   \qquad \quad \> i=1,\dots,k, \\ 
    e_{k+1}^T T_k^C e_{k+1} & =  \frac1{\alpha_{k-1}
  (g_{k}-g_{k-1})^Tp_{k-1}}, \label{eq:Tclast} \\ 
e_{1}^T T_k^{\phi} e_{1} & = - \phi_0 g_0^Tp_0 \left( \frac1{\left( g_1 - g_0 \right)^Tp_0} + \frac1{g_0^Tp_0}\right)^2, \label{eq:Tphifirst} \\
e_{i+1}^T T_k^{\phi} e_{i+1}  &=  -\phi_{i-1} g_{i-1}^Tp_{i-1} \left( \frac1{\left( g_i - g_{i-1} \right)^T p_{i-1}} \right)^2  \nonumber \\ 
& \quad -\phi_{i} g_{i}^Tp_{i} \left( \frac1{\left( g_{i+1} - g_{i} \right)^T p_{i}} + \frac1{g_i^Tp_i} \right)^2,  \quad \> i=1,\dots,k-1, \\
e_{i+1}^T T_k^{\phi} e_{i} &=  e_{i}^T T_k^{\phi} e_{i+1} = \phi_{i-1}   \frac{g_{i-1}^T p_{i-1}}{ \left( ( g_i - g_{i-1})^T p_{i-1} \right)^2}, \nonumber \\
& \quad \quad+  \phi_{i-1}\frac1{  \left( g_i - g_{i-1} \right)^T p_{i-1} }, \qquad \qquad \qquad \quad \> \>  i=1,\dots,k, \\ 
e_{k+1}^T T_k^{\phi} e_{k+1} &=   -\phi_{k-1} g_{k-1}^Tp_{k-1} \left( \frac1{\left( g_k - g_{k-1} \right)^T p_{k-1}} \right)^2 .
\label{eq:Tphilast}
\end{align}
\end{subequations}\endgroup
\end{lemma}
\begin{proof}
The result follows directly from telescoping (\ref{eq:broydenFam}) and
writing it on outer product form. 
\end{proof}

The compact representation in Lemma~\ref{lemma:CRall} requires storage of ($k$+1) gradient vectors and an explicit component matrix, $T_k$, of size ($k$+1)$\times$($k$+1). In comparison to compact representations given in \cite{ByrNocSna94}, \cite{ErwMar14} and \cite{ErwMar17} which require storage of $2k$ vector-pairs $\left( B_0 s_i, \> y_i \right)$, $i=0, \dots, k-1$, and an implicit $2k$$\times $$2k$ component matrix. 
An expression for the inverse of the proposed representation (\ref{eq:cmpRep:outerProdForm}) can be obtained with Sherman-Morrison-Woodbury formula \cite{GolVan13}.

One of the most commonly used quasi-Newton update schemes is the BFGS
update. %, i.e., the update where $B_k$ takes the form
%(\ref{eq:broydenFam}) with $\phi_{k-1}=0$.  We will put a particular
%focus on this update in the remainder of this section and refer to
%quantities $B_k$, $p_k$ and $\alpha_k$ corresponding to this case as
%$B_k^{BFGS}$, $p_k^{BFGS}$ and $\alpha_k^{BFGS}$. 
For $\phi_i = 0$, $i=1,\dots, k-1$, it follows from Lemma~\ref{lemma:CRall} that
\begin{equation*} % \label{eq:CRbfgs}
B_k^{BFGS} = B_0 +\sum_{i=0}^{k-1} \left[{\textstyle\frac1{g_{i}^T p_i}}g_i g_i^T +
  {\textstyle\frac1{\alpha_i
    (g_{i+1}-g_i)^Tp_i}}(g_{i+1}-g_i)(g_{i+1}-g_i)^T\right],
\end{equation*}
or equivalently 
\begin{equation*}% \label{eq:CRbfgsOPF}
B_k^{BFGS} = B_{0} + G_k T_k^{BFGS} G_k^T,
\end{equation*}
where $G_k = \begin{bmatrix}
g_{0} & g_1 & \dots & g_{k-1} & g_k 
\end{bmatrix} \in \mathbb{R}^{n \times (k+1)}$
and $T_k^{BFGS} \in \mathbb{R}^{(k+1) \times (k+1)}$  is a symmetric tridiagonal matrix with elements given in (\ref{eq:Tcfirst})-(\ref{eq:Tclast}). 

Under exact linesearch the step length is chosen such that
$g_{k}^Tp_{k-1} = 0.$
%i.e. $\alpha_{k-1}$ is chosen as the step length to a stationary
%point along $p_{k-1}$.
In consequence the rank-one matrix $ \phi_{k-1} \omega_{k-1}\omega_{k-1}^T$ in (\ref{eq:broydenFam}) reduces to
\begin{equation*} %\label{eq:rank1ELS}
 \phi_{k-1} \omega_{k-1} \omega_{k-1}^T =  - \frac{\phi_{k-1}}{g_{k-1}^Tp_{k-1}}   g_kg_k^T.
\end{equation*}
The choice of Broyden member is thus only reflected in the diagonal of $T_k$ in
Lemma~\ref{lemma:CRall}. This can be observed directly in
(\ref{eq:Tphifirst}) - (\ref{eq:Tphilast}) by making use of the exact
linesearch condition $g_i^Tp_{i-1}=0$, $i=1,\dots,k$. All non-diagonal
terms of $T_k^{\phi}$ become zero and the diagonal terms may be
simplified to
\begin{equation*}
e_{i+1}^T T^{\phi}_k e_{i+1} =  
\begin{cases}
0 & i = 0, \\
-\frac{\phi_{i-1}}{g_{i-1}^Tp_{i-1}} & i=1,\dots,k.
\end{cases}
\end{equation*}
Any Hessian approximation in the Broyden class may in fact be written
as \linebreak $B_k = B_k^{BFGS} - \left( \phi_{k-1} / g_{k-1}^Tp_{k-1} \right)
g_kg_k^T$. In consequence, $B_k$ is independent of $\phi_i$ for $i= 0,
\dots, k-2$ and the choice of Broyden member only affects the scaling
of the search direction. This property follows solely from exact
linesearch, in comparison to the properties that stem from exact
linesearch on quadratic optimization problems (\ref{eq:QP}), which are
discussed in Section~\ref{section:QP}. 
% \textbf{Changes made to the paragraph above}

%This result is not new, however an addition to this and an alternative proof using the proposed compact representation is given in Appendix, Lemma \ref{lemma:UpdBroyFamELS}. The result is given to emphasize the properties that follow solely from exact linesearch. In comparison to the properties that stem from exact linesearch on quadratic optimization problems (\ref{eq:QP}), which are discussed in Section~\ref{section:QP}. 

\section{Quadratic problems} \label{section:QP}
In this section we consider quadratic problems on the form (\ref{eq:QP}) and start from the requirement that $p_k$ generated by the exact linesearch method of Algorithm~\ref{alg:ELSMQP} shall be parallel to $p_k^{PCG}$. Motivated by the performance of the Broyden class, we start by considering Hessian approximations $B_k=B_{k-1}+U_k$ where $U_k$ is a symmetric rank-two matrix with $\mathcal{R}(U_k) = \textit{span}\left(\{ g_{k-1}, g_k \}\right)$ and thereafter look at generalizations. A characterization of all such update matrices $U_k$ is provided as well as a multi-parameter Hessian approximation that generates $p_k = \delta_k p_k^{PCG}$ for nonzero scalar a $\delta_k$. Thereafter, we consider limited-memory Hessian approximations with this property, discuss potential extensions and how to solve the arising systems with a reduced-Hessian method.
\begin{proposition} \label{prop:charUk} 
Consider iteration $k$, $1\leq k < r$, of the exact linesearch method of Algorithm~$\ref{alg:ELSMQP}$ for solving $(\ref{eq:QP})$. Assume that $p_i=\delta_i p_i^{PCG}$ with $\delta_i \neq 0$ for \linebreak$i= 0, \dots , k-1$, where $p_i^{PCG}$ is the search direction of PCG with associated $B_0 \in \mathcal{S}_+^n$, as stated in $(\ref{eq:pPCG})$. Let $B_{k-1}$ be a nonsingular matrix such that $B_{k-1}p_{k-1}=-g_{k-1}$ and $B_{k-1}B_0^{-1}g_k=g_k$. Moreover, let $U_k = B_k - B_{k-1}$ and assume that $B_k$ and $p_k$ satisfy $B_kp_k = -g_k$ with $B_k$ nonsingular. Then, if $U_k$ is symmetric, rank-two with $\mathcal{R}(U_k) = \textit{span}\big(\{ g_{k-1}, g_k \}\big)$ it holds that $p_k=\delta_k p_k^{PCG}, \delta_k \neq 0$, if and only if
\[
U_k =\frac1{g_{k-1}^T p_{k-1}} g_{k-1} g_{k-1}^T
 +  \rho_{k-1} \left( g_k - g_{k-1} \right)\left( g_k - g_{k-1}
\right)^T + \frac{\left( \frac{1}{\delta_k}-1 \right)}{g_k^T B_0^{-1} g_k} g_{k} g_{k}^T,
\]
where $\rho_{k-1}$ is a free parameter. 
\end{proposition}
\begin{proof}
As stated in \cite[Proposition 5]{ForOdl18}, the assumptions in the proposition together with (\ref{eq:iffCondOnBk}) and $B_k=B_{k-1}+U_k$ give the following necessary and sufficient condition on $U_k$ such that $p_k=\delta_k p_k^{PCG}$ for a scalar $\delta_k \neq 0$,
\begin{equation} \label{eq:UkCond}
U_k \left(B_0^{-1}g_k + \frac{g_k^T B_0^{-1} g_k}{p_{k-1}^Tg_{k-1}} p_{k-1}\right)  =  \left( \frac{1}{\delta_k}-1 \right)g_k + \frac{g_k^T B_0^{-1} g_k}{p_{k-1}^Tg_{k-1}}g_{k-1}.
\end{equation}
Any symmetric rank-two matrix, $U_k$, with  $\mathcal{R}(U_k) = \textit{span}\big(\{ g_{k-1}, g_k \}\big)$ can be written as
\begin{equation} \label{eq:AllUkQP}
U_k = \begin{pmatrix}
g_{k-1} & g_{k} 
\end{pmatrix} \begin{pmatrix}
\eta_{k-1} +  \rho_{k-1} & - \rho_{k-1} \\ - \rho_{k-1} & \varphi_k + \rho_{k-1}
\end{pmatrix}  \begin{pmatrix}
g_{k-1}^T \\ g_{k}^T 
\end{pmatrix},
\end{equation}
for parameters $\eta_{k-1}$, $\rho_{k-1}$ and
$\varphi_{k}$. Insertion of (\ref{eq:AllUkQP}) into (\ref{eq:UkCond}),
taking into account $g_k^T B_0^{-1} g_{k-1} = 0$ and $g_k^T p_{k-1}=
0$ gives
\[
\varphi_k g_k^T B_0^{-1} g_k g_k + \eta_{k-1} g_k^T B_0^{-1} g_k g_{k-1} = \Big( \frac{1}{\delta_k}-1 \Big)g_k + \frac{g_k^T B_0^{-1} g_k}{p_{k-1}^Tg_{k-1}}g_{k-1},
\]
which is independent of $\rho_{k-1}$. Identification of coefficients
for $g_k$ and $g_{k-1}$ respectively gives
\[
\varphi_{k}  = \Big( \frac{1}{\delta_k}-1 \Big)\frac{1}{g_k^T B_0^{-1}
  g_k}, \quad  \text{ and } \quad  \eta_{k-1}  =  \frac{1}{g_{k-1}^Tp_{k-1}}.
\]
Insertion of $\varphi_{k}$ and $\eta_{k-1}$ into (\ref{eq:AllUkQP})
gives $U_k$ as stated in the lemma, with $\rho_{k-1}$ free. 
\end{proof}

The result in Proposition~\ref{prop:charUk} provides a two-parameter update matrix, $U_k$. If the conditions of Proposition~\ref{prop:charUk} apply then it follows directly from $U_k$ that the iterates satisfy the scaled secant condition (\ref{eq:scaledSecCond}). This can be seen by considering $B_k \alpha_{k-1} p_{k-1}$ with $B_k = B_{k-1} + U_k$ which gives \[
\left( B_{k-1}+U_k \right) \alpha_{k-1} p_{k-1} =  -\rho_{k-1} \alpha_{k-1}  g_{k-1}\T p_{k-1}\left( g_k - g_{k-1} \right).
\] Consequently the characterization in Proposition~\ref{prop:charUk}
provides a class which under exact linesearch is equivalent to the symmetric Huang class. The
scaling in the secant condition does neither affect the search
direction nor its scaling. Utilizing the secant condition sets
the parameter $\rho_{k-1} = -1/ \left( \alpha_{k-1}g_{k-1}^Tp_{k-1}
\right)$, that together with the change of variable 
\begin{equation} \label{eq:1to1}
\varphi_k =
\left( \frac{1}{\delta_k}-1 \right)\frac{1}{g_k^TB_0^{-1}g_k} =
-\frac{\phi_{k-1}}{g_{k-1}^Tp_{k-1}},
\end{equation} gives the exact linesearch
form of the Broyden class matrices in (\ref{eq:broydenFam}). Hence, as
expected, utilizing the secant condition fixates one of the parameters
and gives the Broyden class. 
% Note that $U_k$ may equivalently be written as
% \begin{equation} \label{eq:AllUkQPII}
% U_k = \begin{pmatrix}
% g_{k-1} & g_{k} 
% \end{pmatrix} \begin{pmatrix}
% \frac1{g_{k-1}^T p_{k-1}} + \rho_{k-1} & - \rho_{k-1} \\ -
% \rho_{k-1} & 
% \Big( \frac{1}{\delta_k}-1 \Big)\frac{1}{g_k^T B_0^{-1}
%   g_k} + \rho_{k-1}
% \end{pmatrix}  \begin{pmatrix}
% g_{k-1}^T \\ g_{k}^T 
% \end{pmatrix},
% \end{equation}
% which makes it straightforward to see why the Broyden symmetric
% rank-one update breaks down for the unit steplength. At iteration $k$
% when $\alpha_{k-1}=1$, then $\rho_{k-1} = - 1 / g_{k-1}^Tp_{k-1}$ so
% that $1 / g_{k-1}^Tp_{k-1} +\rho_{k-1}=0$. In this situation, $U_k$ of
% (\ref{eq:AllUkQPII}) is indefinite, since a symmetric two-by-two
% matrix with a zero on the diagonal and a nonzero on the offdiagonal
% has negative determinant. Hence, $U_k$ has rank two. Moreover, if
% $U_k$ of Proposition~\ref{prop:charUk} is required to be a symmetric
% rank-one matrix, $\rho_{k-1}$ may be eliminated from the update
% formula to give the one-parameter class suggested by Forsgren and
% Odland \cite{ForOdl18}.

The result of Proposition~\ref{prop:charUk}
motivates further study of the structure in the corresponding Hessian
approximations.
\begin{lemma}\label{lemma:SymHuangEquiv}
Consider iteration $k$, $1\leq k < r$, of the exact linesearch method of
  Algorithm~$\ref{alg:ELSMQP}$ for solving $(\ref{eq:QP})$. Assume that $B_i p_i=-g_i$, $i=0,\dots,k-1$, where
  $B_0 \in \mathcal{S}_+^n$ and 
\begin{align}
B_i & =  B_{i-1} + \frac1{g_{i-1}^T p_{i-1}} g_{i-1} g_{i-1}^T
  +  \rho_{i-1} \left( g_i - g_{i-1} \right)\left( g_i - g_{i-1}
\right)^T \nonumber \\
& \quad +\varphi_{i} g_{i} g_{i}^T, \qquad \qquad \qquad \qquad \qquad \qquad \qquad i=1, \dots, k,
\label{eq:BhuangFamQP}
\end{align}
with $\rho_{i-1}$ and $\varphi_{i}$ chosen such that $B_i$ is nonsingular. Then $B_k$ takes the form
\begin{align}
B_k  & = B_0 +\sum_{i=0}^{k-1} \left(-\frac1{g_{i}^TB_0\inv g_i}g_i
  g_i^T +
  \rho_i (g_{i+1}-g_i)(g_{i+1}-g_i)^T\right) + \varphi_{k} g_{k} g_{k}^T.
\label{eq:BequivHuangQP} 
\end{align}
\end{lemma}
\begin{proof}
With the assumptions in the proposition, the update of (\ref{eq:BhuangFamQP}) satisfies the requirements of Proposition~\ref{prop:charUk} and hence for each $i$, $i = 0,\dots, k-1$, it follows that $p_i = \delta_i p_i^{PCG}$ where $\delta_i = 1/\left( 1+\varphi_{i} g_{i}^T B_0\inv g_{i} \right)$ and $\varphi_0 = 0$. Premultiplication of $p_i = \delta_i p_i^{PCG}$ by $g_i^T $ gives
\begin{equation}
g_i^T p_i = \frac{1}{1+\varphi_i g_i^T B\inv _0  g_i} g_i^T p_i^{PCG}, \qquad i = 0, \dots, k-1,
\label{eq:gTpPARAgTpcg}
\end{equation}
Inverting (\ref{eq:gTpPARAgTpcg}) and taking into account that $g_i^T p_i^{PCG} = -g_i^T B\inv _0 g_i$, $i=0,\dots, k-1$, gives 
\begin{equation}\label{eq:gTpPphi}
\frac1{g_{i}^T p_{i}} + \varphi_{i} = -\frac1{g_{i}\T B_0\inv g_{i}},  \qquad i=0,\dots, k-1.
\end{equation}
By telescoping (\ref{eq:BhuangFamQP}) at iteration $k$ we obtain
\begin{align}
B_k & = B_0 +\sum_{i=0}^{k-1} \left[ \left( \frac1{g_{i}^T p_i} + \varphi_i \right) g_i
  g_i^T +
  \rho_i (g_{i+1}-g_i)(g_{i+1}-g_i)^T\right] \nonumber \\ 
  & \quad + \varphi_{k} g_{k} g_{k}^T.
\label{eq:teleHuangQP}
\end{align}
Insertion of (\ref{eq:gTpPphi}) into (\ref{eq:teleHuangQP}) gives (\ref{eq:BequivHuangQP}). 
\end{proof}
Lemma~\ref{lemma:SymHuangEquiv} and (\ref{eq:BequivHuangQP}) show that if $p_k$ is given by (\ref{eq:QN}) with $B_k$ as in (\ref{eq:BequivHuangQP}), then $p_k$ is independent of all $\rho_{i}$, $i = 0, \dots, k-1$, as long as $B_k$ is nonsingular. This result is formalized in the following proposition. 

\begin{proposition}\label{prop:genHuangQP}
  Consider iteration $k$, $1\leq k < r$, of the exact linesearch method of Algorithm~$\ref{alg:ELSMQP}$ for solving $(\ref{eq:QP})$. Assume that
  $p_i=\delta_i p_i^{PCG}$ with $\delta_i \neq 0$ for \linebreak $i= 0, \dots ,
  k-1$, where $p_i^{PCG}$ is the search directions
  of PCG with associated $B_0 \in \mathcal{S}^n_+$, as stated in
  $(\ref{eq:pPCG})$. Let $p_k$ satisfy $B_kp_k = -g_k$ where
\begin{align}
B_k & = B_0 +\sum_{i=0}^{k-1} \left(\frac{-1}{g_{i}^TB_0\inv g_i}g_i
  g_i^T +
  \rho_i^{(k)} (g_{i+1}-g_i)(g_{i+1}-g_i)^T\right) 
+ \varphi_{k} g_{k} g_{k}^T, 
\label{eq:BMuP}
\end{align}
with $\rho_i^{(k)}$, $i=0,\dots,k-1$, and $\varphi_{k}$ chosen such
that $B_k$ is nonsingular. Then,
\[%\label{eq:pPar2pPCG}
p_k=\frac1{1+\varphi_{k} g_{k}^T B_0\inv g_{k}}
p_k^{PCG}.
\]
In particular, if  $\rho_i^{(k)} > 0$, $i=0,\dots,k-1$, and $\varphi_{k}>-1/(g_{k}^T B_0\inv g_{k})$, then $B_k\succ0$.
\end{proposition}
\begin{proof}
From Proposition~\ref{prop:charUk} and Lemma~\ref{lemma:SymHuangEquiv} it follows that $B_k$ given by (\ref{eq:BequivHuangQP}) generates $p_k = \delta_k p_k^{PCG}$ where $\delta_k = 1/ \left(1+\varphi_{k} g_{k}^T B_0\inv g_{k}\right)$ and hence satisfies
\[
(g_{i+1}-g_i)^T p_k = 0, \qquad i=0,\dots, k-1,
\]
by Lemma~\ref{lemma:conjProp}. If $\rho_i$, $i=0,\dots,k-1$, and $\varphi_{k}$ chosen such that $B_k$ is nonsingular then the solution is unique and independent of $\rho_i$, $i=0,\dots, k-1$, and thus $\rho_i = \rho_i^{(k)}$, $i=0,\dots, k-1$. Moreover, if $\rho_i^{(k)} > 0$, $i=0,\dots,k-1$ and $\varphi_k = 0$ then the matrix of (\ref{eq:BMuP}) is positive definite by Lemma~\ref{Ap:lemma:Bposdef}. It then follows from Lemma~\ref{Ap:lemma:Axb} that $B_k\succ 0$ for $\varphi_{k}>-1/(g_{k}^T B_0\inv g_{k})$. 
\end{proof}
The result in Proposition~\ref{prop:genHuangQP} together with the exact linesearch method of Algorithm~\ref{alg:ELSMQP} provide a multiple-parameter class that generates search directions parallel to those of PCG. In the framework of updates on the form $B_k = B_{k-1}+U_k$ this class allows update matrices with $\mathcal{R}(U_k) = \textit{span}\left(\{ g_{0},\dots, g_k \}\right)$ and reduces to the symmetric Huang class if $\mathcal{R}(U_k) = \textit{span}\left(\{ g_{k-1}, g_k \}\right)$ is required. The Hessian approximation in (\ref{eq:BMuP}) of Proposition~\ref{prop:genHuangQP} can also be viewed as a matrix composed of three parts
\begin{equation} \label{eq:Bk3parts}
B_k = B_0 + V_k^{\varphi} +  V_k^{\rho},
\end{equation}
where
\begin{equation*}
V_k^{\varphi} = \sum_{i=0}^{k-1} \frac{-1}{g_{i}^TB_0\inv g_i}g_i g_i^T + \varphi_k g_k g_k^T,
\end{equation*}
 and
 \begin{equation} \label{eq:Vrho}
  V_k^{\rho} = \sum_{i=0}^{k-1} \rho_i^{(k)} (g_{i+1}-g_i)(g_{i+1}-g_i)^T.
 \end{equation}
 The direction is determined solely by $B_0 + V_k^{\varphi}$, compare
 with (\ref{eq:pPCGexp}). The parameter $\varphi_k$ gives a scaling
 and the parameters $\rho_{i}^{(k)}$, $i = 0, \dots, k-1$, have no
 effect on the direction. Certain choices of these parameters merely
 guarantee nonsingularity and may provide numerical
 stability. 
\subsection{Limited-memory Hessian approximations}
In this section we extend the above discussion to limited-memory
Hessian approximations. The goal is to obtain approximations such that
(\ref{eq:QN})  gives search directions parallel to those of PCG. From
(\ref{eq:pPCGexp}) it follows directly that $p_k^{PCG} \in \textit{span}(\{ B\inv _0 g_0, \dots, B\inv _0 g_k\})$ and that $p_k^{PCG}$ has a nonzero component in every direction $B\inv _0 g_i$, $i=0,\dots, k$. In consequence, gradient information can not be discarded if $B_k$ is on the form of (\ref{eq:Bk3parts}) where $\mathcal{R}(V_k^{\varphi} +  V_k^{\rho} ) = \textit{span}(\{ g_0, \dots, g_k\})$.  As long as $B_k$ remains nonsingular, gradient information can be discarded from $V_k^\rho$ but not from the part essential for the direction, namely $V_k^{\varphi}$. 
However, parallel directions can be generated if a specific correction term is added to the right hand side of (\ref{eq:QN}), as is shown in Appendix, Theorem~\ref{thm:LMCgrad}. It can also be done as e.g.~in \cite{ByrNocSna94}, by at each iteration $k$ recalculating the basis vectors from the $m$ latest vector pairs $\left( s_i,\> y_i\right)$, $i = k-m,\dots, k-1$. 

In light of the above, the discussion will now be extended to consider Hessian approximations on the form $B_k = B_0 + V_k + V_k^{\rho}$ where $V_k$ is not restricted to have an outer-product form consisting of only gradients. Theorem~\ref{thm:LMQN} below provides conditions on $V_k$ and $V_k^{\rho}$ such that $B_k$ is positive definite and generates directions parallel to those of PCG.  
\begin{theorem} \label{thm:LMQN} Consider iteration $k$, $1\leq k <
  r$, of the exact linesearch method of Algorithm~$\ref{alg:ELSMQP}$
  for solving $(\ref{eq:QP})$. Assume that $p_i=\delta_i p_i^{PCG}$
  with $\delta_i \neq 0$ for \linebreak $i= 0, \dots , k-1$, where $p_i^{PCG}$ is
  the search direction of PCG with associated $B_0 \in
  \mathcal{S}_+^n$, as stated in $(\ref{eq:pPCG})$. Let $p_k$ satisfy
  $B_kp_k = -g_k$ with $B_k = B_0 + V_k + V_k^\rho$ where $V_k^\rho$
  is defined by $(\ref{eq:Vrho})$ and $V_k$ is symmetric such that
  $B_0 + V_k \succeq 0$. If, in addition, $\rho_i^{(k)} \geq 0$,
$i = 0, \dots, k-1$, and $\mathcal{N}(B_0 + V_k) \cap
\mathcal{N}(V_k^\rho) = \emptyset$, then  $B_k \succ 0$, and, for
$\delta_k\ne 0$, it holds that $p_k = \delta_k p_k^{PCG}$ if and only if
\begin{equation}\label{eq:VkCond}
  V_k \left(B_0^{-1}g_k + \frac{g_k^T B_0^{-1} g_k}{p_{k-1}^Tg_{k-1}} p_{k-1}\right)  =  \left( \frac{1}{\delta_k}-1 \right)g_k - \frac{g_k^T B_0^{-1} g_k}{p_{k-1}^Tg_{k-1}}B_0 p_{k-1}.
\end{equation}
In particular, $(\ref{eq:VkCond})$ with $\delta_k=1$, $B_0 + V_k \succeq 0$ and $\mathcal{N}(B_0 + V_k) \cap
\mathcal{N}(V_k^\rho) =\emptyset$ are satisfied for 
\begin{equation}\label{eq:symPCG}
V_k = C_k^T B_0 C_k - B_0, \text{with $C_k$ defined as in $(\ref{eq:Ck})$,} 
\end{equation}
 and
\begin{equation} \label{eq:Vsr1}
V_k = - \frac{1}{p_{k-1}^T B_0 p_{k-1}} B_0 p_{k-1} p_{k-1}^T B_0,  \text{with $\rho^{(k)}_{k-1}>0$}. 
\end{equation}
\end{theorem}
\begin{proof}
  The parameters $\rho_i^{(k)} \geq 0$, $i=0,\dots, k-1$ give
  $V_k^\rho \succeq 0$. It then follows from $\mathcal{N}(B_0 + V_k)
  \cap \mathcal{N}(V_k^\rho) = \emptyset$ that $B_k \succ
  0$. Insertion of $B_k=B_0+V_k+V_k^\rho$ into (\ref{eq:iffCondOnBk}) gives the if and only if condition
  \begin{equation} \label{eq:thm:proofcond1}
  (B_0+V_k+V_k^\rho)
\left(B_0^{-1}g_k + \frac{g_k^T B_0^{-1} g_k}{p_{k-1}^Tg_{k-1}}
  p_{k-1}\right)  =  
\frac{1}{\delta_k}g_k,
  \end{equation}
 for $p_k = \delta_k p_k^{PCG}$, $\delta_k \neq 0$. Using the identity
\[
V_k^\rho
\left(B_0^{-1}g_k + \frac{g_k^T B_0^{-1} g_k}{p_{k-1}^Tg_{k-1}}
  p_{k-1}\right)  =  0,
\]
which follows from Lemma~\ref{lemma:conjProp}, in (\ref{eq:thm:proofcond1}) and moving terms corresponding to $B_0$ to the right-hand side gives if and only if condition (\ref{eq:VkCond}) on $V_k$. In particular, with $B_0+V_k=C_k^T B_0 C_k$, as given by
(\ref{eq:symPCG}), letting $W_k=B_0$ in (\ref{eq:iffCondOnBk2}) gives
$p_k=p_k^{PCG}$. In addition, since $C_k$ is nonsingular and $B_0\succ
0$, it follows that $B_0+V_k = C_k^T B_0 C_k\succ 0$, so that $B_0+V_k+V_k^\rho
\succ 0$ for $\rho_i^{(k)} \geq 0$, $i=0,\dots, k-1$. The null-space of $B_0+V_k$ with $V_k$ as in (\ref{eq:Vsr1}) is
one-dimensional and spanned by $p_{k-1}$. Since $B_0\succ 0$, we have
$B_0+V_k\succeq 0$. In addition, if $\rho^{(k)}_{k-1}> 0$, then
\[
V_k^\rho p_{k-1}=\rho^{(k)}_{k-1}(g_k-g_{k-1})\T p_{k-1} (g_k-g_{k-1}) = 
-\rho^{(k)}_{k-1} g_{k-1}^T p_{k-1} (g_k-g_{k-1}) \ne 0,
\]
since $g_{k-1}^T p_{k-1}^{PCG} < 0$ and $p_{k-1}=\delta_{k-1}
p_{k-1}^{PCG}$, with $\delta_{k-1}\ne 0$, by assumption. Therefore,
$B_0+V_k+V_k^\rho \succ 0$ if $\rho^{(k)}_{k-1}> 0$ and $\rho_i^{(k)}
\geq 0$, $i=0,\dots, k-2$. Finally, $V_k$ of (\ref{eq:Vsr1})
satisfies (\ref{eq:VkCond}) for $\delta_k=1$ since $g_k^T p_{k-1}=0$,
as required. 
\end{proof}
The result in Theorem~\ref{thm:LMQN} provides a class of
multi-parameter limited-memory Hessian approximations where the memory
usage can be changed between iterations. The choices of $V_k$ in
(\ref{eq:symPCG}) and (\ref{eq:Vsr1}) are merely two examples of
members in the class. The matrix $V_k$ of (\ref{eq:symPCG})
is of rank-two, and if used in $B_k= B_0 + V_k + V_k^{\rho}$, with $\rho_i^{(k)}=0$,
$i=0,\dots, k-1$, then $B_kp_k = -g_k$ may viewed as a symmetric PCG update, compare with (\ref{eq:qNPCG}). The matrix $V_k$ of (\ref{eq:Vsr1}) is the matrix of least rank that satisfies (\ref{eq:VkCond}) with $\delta_k = 1$. In general, there is little restriction on which iterations to include information from in $V_k^\rho$. Information can thus be included from iterations that are believed to be of importance. All information may also be expressed in terms of search directions and the current gradient $g_k$. This provides the ability to reduce the amount of storage when the arising systems are solved by reduced-Hessian methods, described in Section~\ref{subsec:SolvDsyst}, with search directions in the basis.
\subsection{Solving the systems} \label{subsec:SolvDsyst} 
In this section we discuss solving systems of linear equations using
reduced-Hessian methods. These methods provide an alternative
procedure for solving systems arising in quasi-Newton methods. We
follow Gill and Leonard \cite{GilLeo01,GilLeo03} and refer to their
work for a thorough introduction.

Assume that a Hessian approximation of Theorem~\ref{thm:LMQN} is given and used together with the exact linesearch method of Algorithm~$\ref{alg:ELSMQP}$ for solving $(\ref{eq:QP})$. The search direction at iteration $k$ then satisfies $p_k=\delta_k p_k^{PCG}$ for a scalar $\delta_k$ and hence by (\ref{eq:pPCG}) $p_k\in\textit{span}\left( \{ p_{k-1}, B\inv _0 g_k \} \right)$. Define $\Psi^{\min}_k = \{ p_{k-1}, B\inv _0 g_k \}$ and let $\Psi_k$ be a subspace such that $ \Psi^{\min}_k \subseteq \Psi_k$. Furthermore let $Q_k$ be a matrix whose columns span $\Psi_k$ and $Z_k$ be the matrix whose columns are the vectors obtained from the Gram-Schmidt process on the columns of $Q_k$. The search direction can then be written as $p_k = Z_k u_k$ for some vector $u_k$. Premultiplication of (\ref{eq:QN}) by $Z_k^T$ together with $p_k = Z_k u_k$ gives
\begin{equation} \label{eq:redHess}
Z_k^T B_k Z_k u_k = -Z_k^T g_k,
\end{equation}
which has a unique solution if $B_k$ is positive definite. Hence $p_k = Z_k u_k$ where $u_k$ satisfies (\ref{eq:redHess}). Note that the analogous procedure is also applicable for the result in Appendix, Theorem~\ref{thm:LMCgrad} where the Hessian approximation is given by (\ref{eq:BgenLC}) and $p_k$ is generated by $B_k p_k = -N_k g_k$. The minimal space required is $\Psi_k = \Psi^{\min}_k$ but other feasible choices are for example $\Psi_k=\{ B\inv _0 g_0, \dots, B\inv _0 g_k\}$, by (\ref{eq:pPCGexp}), or $\Psi_k=\{p_{t-1},B\inv _0g_t,\dots,B\inv _0g_k\}$ where $0 < t < k$. 
\subsection{Construction of methods and complexity} \label{subsec:methodsAndComplexity}
The Hessian approximations proposed in Theorem~\ref{thm:LMQN} combined with the reduced-Hessian framework of Section~\ref{subsec:SolvDsyst} provide freedom in the construction of limited-memory quasi-Newton methods. A summary of the quantities that can be chosen is shown in Table~\ref{table:MethodPar} below.
  \begin{table}[H]
 % \captionsetup{font=footnotesize,labelsep=colon,labelfont=bf} 
      \centering
      \caption{Variable quantities at iteration $k$.}
      \begin{footnotesize}
\begin{tabular}{c|l} \label{table:MethodPar}
Quantity & Description  \\ \hline\noalign{\smallskip}
$V_k$ & Must satisfy the conditions of Theorem~\ref{thm:LMQN}.  \\
$\rho_i^{(k)}$,  $i=0,\dots, k-1$ & Parameter values.   \\
$m_k$ & $\#$ non-zero $\rho_i^{(k)}$, $i=0,\dots, k-1$, of $V_k^{\rho}$ in (\ref{eq:Vrho}).  \\
$Q_k $ & Provides the space for $p_k$, columns must span $ \Psi_k^{\min} $.\\
$\hat{m}_k$ &  $\#$  columns of $Q_k$ and dimension of the reduced-Hessian.
      \end{tabular}
            \end{footnotesize}
\end{table} 
\noindent
The complexity of the method is essentially determined by the construction and solution of (\ref{eq:redHess}). Each iteration $k$ requires a Gram-Schmidt process as well as the construction and factorization of the matrix $Z_k^T \left( B_0 + V_k + V_k^{\rho} \right) Z_k$. If no information is re-used this gives the worst case complexity 
\[
O(n \hat{m}_k^2 + \left[ ( n^2 \hat{m}_k + n \hat{m}_k^2) +
  \mbox{rank}(V_k) (n \hat{m}_k + \hat{m}_k^2) + (nm \hat{m}_k + m \hat{m}_k^2) \right] + \hat{m}_k^3), 
\]
where constants have been omitted. However, the overall complexity can be reduced with particular choices of the quantities in Table~\ref{table:MethodPar}. We will demonstrate this by constructing two relatively simple quasi-Newton methods with fixed limited-memory $m=\hat{m}>3$ for which numerical results are given in Section~\ref{section:numRes}. The methods will be denoted by \texttt{symPCGs} and  \texttt{Vsr1} for which $V_k$ is given by (\ref{eq:symPCG}) of Theorem~\ref{thm:LMQN} and (\ref{eq:Vsr1}) respectively. Both use Hessian approximations on the form $B_k = B_0 + V_k + V_k^{\rho}$ with $V_k^{\rho}$ as in (\ref{eq:Vrho}) and parameters defined by
\[\rho_i^{(k)} = \begin{cases} \rho_i^B &  i = 0, \dots, m-4, k-3, k-2, k-1, \\
0 & i = m-5, \dots, k-4, \end{cases} \]
where $\rho_i^B$, $i = 0, \dots, k-1$, are the quantities which corresponds to the secant condition. The basis is computed from \[ Q_k = \begin{pmatrix}
p_0 & \dots & p_{m-4} & p_{k-2} & p_{k-1} & B\inv _0 g_k
\end{pmatrix}.\] Hence systems of size at most $m$ has to be solved at every iteration with information from the ($m$--3)-first and the $3$-latest iterations for $k>m$. In consequence part of the matrices $Z_k$ and $V_k^{\rho}$ stay constant and the reduced-Hessian may be updated instead of recomputed at every iteration. The computational complexity is then dominated by $\max\{ n^2, m^3 \}$, see Appendix for motivation. The choice $m = n^{2/3}$ gives complexity $n^2$ which is the same as that of PCG with exact linesearch.
 
\section{Numerical results} \label{section:numRes} In this section we
first show the convergence behavior of quasi-Newton-like methods and
PCG for randomly generated quadratic optimization problems.
Thereafter we give performance profiles \cite{DolMor2002} for the two
limited-memory methods described in
Section~\ref{subsec:methodsAndComplexity}. As mentioned in
Section~\ref{subsec:methodsAndComplexity}, these methods are referred
to as \texttt{symPCGs} and \texttt{Vsr1} respectively. These are also
compared to our own Matlab implementations of PCG, with search
direction given by (\ref{eq:pPCG}), BFGS, with search direction given
by (\ref{eq:broydenFamys}) with $\phi_{k-1}=0$ for all $k$, and L-BFGS
with search direction as proposed by Nocedal in \cite{Noc80n}. All
methods use exact linesearch, as defined by
Algorithm~\ref{alg:ELSMQP}, with their particular search direction
update. We refer to our implementation of these methods as PCG, BFGS
and L-BFGS respectively. The performance profiles are in terms of
number of iterations for solving systems of linear equations which
originate from the CUTEst test collection \cite{GouOrbToi15}. The
benchmark problems were initially processed using the \texttt{Julia}
packages \texttt{CUTEst.jl} and \texttt{NLPmodels.jl} by Orban and
Siqueira \cite{OrbSigSmoothOpt2019}.

The purpose of the first part is partly to illustrate the difference between quasi-Newton-like methods and PCG to give an indication of the round-off error effects. Convergence for a member in the class of quasi-Newton Hessian approximations in (\ref{eq:BMuP}) of Proposition~\ref{prop:genHuangQP}, here denoted by \texttt{MuP}, is shown in Figure~\ref{fig:MuPwPCGexact}. The figure also contains the convergence of the BFGS method and PCG in both finite and exact arithmetic, all with $B_0 = I$. In this study we consider exact arithmetic PCG as the original but with 512 digits precision. The parameters of (\ref{eq:BMuP}) were chosen as follows, $\varphi_k= 0$ for all $k$ and \[
\rho_i^{(k)} = \xi_i^{(k)} \rho_i^{B}, \qquad i = 0, \dots, k-1,
\]
where $\xi_i^{(k)}$, $i = 0, \dots, k-1$, are normally distributed random variables and $ \rho_i^B$, $i = 0, \dots, k-1$, are the quantities corresponding to the secant condition. Note that the scaling of $\rho_i^{(k)}$, $i= 0, \dots, k-1$, are randomly changed for every $k$.
\begin{figure}[H]
%\captionsetup{font=footnotesize,labelsep=colon,labelfont=bf} 
    \centering
    \includegraphics[width=0.537\textwidth]{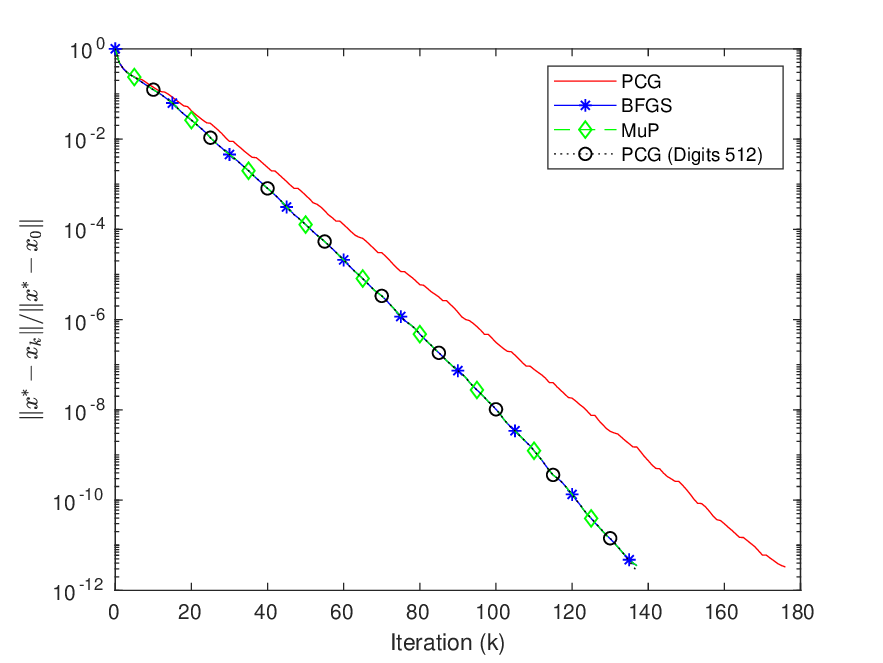}
    \caption{Convergence for solving a randomly generated quadratic problem with $300$ variables and condition number in the order of $10^4$. The convergence corresponds to representative results based on 100 simulations with parameters $\xi_i^{(k)} \in \left[ 10^{-1}, 10^8 \right]$, $i = 0, \dots, k-1$. } \label{fig:MuPwPCGexact}
\end{figure}
\noindent 
The three methods compared in Figure~\ref{fig:MuPwPCGexact} are with the
exact linesearch method of Algorithm~\ref{alg:ELSMQP} equivalent in
exact arithmetic on
(\ref{eq:QP}). However, in finite arithmetic this is not the case. As
can be seen in the figure, PCG suffers from round-off errors while
BFGS behaves like the exact arithmetic PCG. The maximum error from all
simulations between the iterates of BFGS and exact PCG was $ 5.1 \cdot
10^{-14}$, i.e. \[ \max\limits_{i} \| x_i^{BFGS} - x_i^{PCG}\| = 5.1
\cdot 10^{-14}. \] Consequently, the BFGS method does not suffer from
round-off errors on these randomly generated quadratic problems. By the
result of Proposition~\ref{prop:genHuangQP} it is not required to fix
the parameters $\rho_i^{(k)}$, $i=0,\dots, k-1$, and as
Figure~\ref{fig:MuPwPCGexact} shows there is an interval where this
result also holds in finite arithmetic. The secant condition is
expected to provide an appropriate scaling of the quantities since it
gives the true Hessian in $n$ iterations. Our results indicate that
there is no particular benefit for the quadratic case to choose the
values given by the secant condition. This freedom may be useful, since
values of $\rho_i^{(k)}$ close to zero for some $i$ may make the
Hessian approximation close to singular, and such values could
potentially be avoided.

The purpose of the next part is to give a first indication of the
performance of the proposed class of limited-memory quasi-Newton
methods for solving a sequence of related systems of linear
equations. The sequences of systems were generated by considering
unconstrained nonlinear CUTEst problems of the order $10^3$. For such
a problem, a sequence of matrix and right-hand side pairs
$\nabla^2f(x_j)$ and -$\nabla f(x_j)$, \linebreak$j=1,\dots,J$, was accepted if
an initial point $x_0$ for which Newton's method with unit step
converged to a point $x_J$ such that $\| \nabla f(x_J) \| \leq
10^{-6}$ and \linebreak$\lambda_{min}\left( \nabla^2 f(x_J) \right)~>~10^{-6}$. In addition, it was required for each $j$ that \linebreak $\nabla^2
f(x_j) d_j = - \nabla f(x_j)$, was solvable with accuracy at least
$10^{-7}$ by the built-in solver in \texttt{Julia} and that $\nabla^2
f(x_j)$ was positive definite.  These
conditions reduced the test set to $21$ problems giving $21$ sequences
of linear equations, corresponding to $220$ systems of linear
equations in total. The performance measure in terms of number of
iterations is defined as follows. Let $N_{p,s}$ be the number of
iterations required by method $s \in \mathcal{S}$ on problem $p \in
\mathcal{P}$. If the method failed to converge within a maximum number
of iterations this value is defined as infinity. The measure
$P_s(\tau)$ is defined as \begin{equation*} P_s(\tau) = \frac{| \{ p
    \in \mathcal{P} : r_{p,s} \leq \tau \} |}{ | \mathcal{P} |},
  \text{ where } r_{p,s} = \frac{N_{p,s}}{\min\{\ N_{p,s} : s \in
    \mathcal{S} \}},
\end{equation*}
Performance profiles are shown in Figure~\ref{fig:PPwPC} for \texttt{SymPCGs}, \texttt{Vsr1}, PCG, BFGS, and L-BFGS. All with $B_0 = \nabla^2 f(x_{j-1})$, i.e. when the Newton system at iteration $j$ is preconditioned with the previous Hessian. The figure contains results for \linebreak $m = [n^\frac{1}{2}, \> n^\frac{2}{3}, \> n^\frac{21}{30}]$ and two different tolerances in the stopping criterion. %, these will be refered to as low and high accuracy respectively in Table~\ref{table:aveNumIt} and \ref{table:aveSizeSolv} . 
The first corresponds to the criterion $\| g_k \| \leq 10^{-7}$ , namely when an approximate solution is desired. The other corresponds to when a more accurate solution is desired and the criterion $\| g_k \| \leq \max \{\epsilon^{DS}, 10^{-13}\}$, where $ \epsilon^{DS}$ is the solution accuracy of the built-in direct solver on the preconditioned Newton system with preconditioner $B_0$. The two stopping criteria will be referred to as low and high accuracy respectively in Table~\ref{table:aveNumIt}-\ref{table:aveSizeSolv}. For the two first $m$-values the overall computational complexity is $O(n^2)$ whereas for the third it is $O(n^\frac{21}{10})$, i.e. a case where more computational work is allowed if needed. In Table~\ref{table:aveSizeSolv} we show a measure of the mean computational work done for solving systems on the form (\ref{eq:redHess}) per iteration. The measure is in terms of the average size of the reduced-Hessian relative to $m$. Moreover, the average number of iterations are shown in Table~\ref{table:aveNumIt}. The maximum number of iterations was set to $5n$ in all simulations.

We also give performance profiles when no regard is taken to the fact that the systems of linear equations within a sequence are related, namely when $B_0=I$. The corresponding performance profiles and averaged quantities are shown in Figure~\ref{fig:PPwoPC} and Table~\ref{table:aveNumIt}-\ref{table:aveSizeSolv} respectively.

\begin{figure}[H]
%\centering
%\captionsetup{font=footnotesize,labelsep=colon,labelfont=bf} 
\parbox{5.9cm}{
\includegraphics[width=6.5cm]{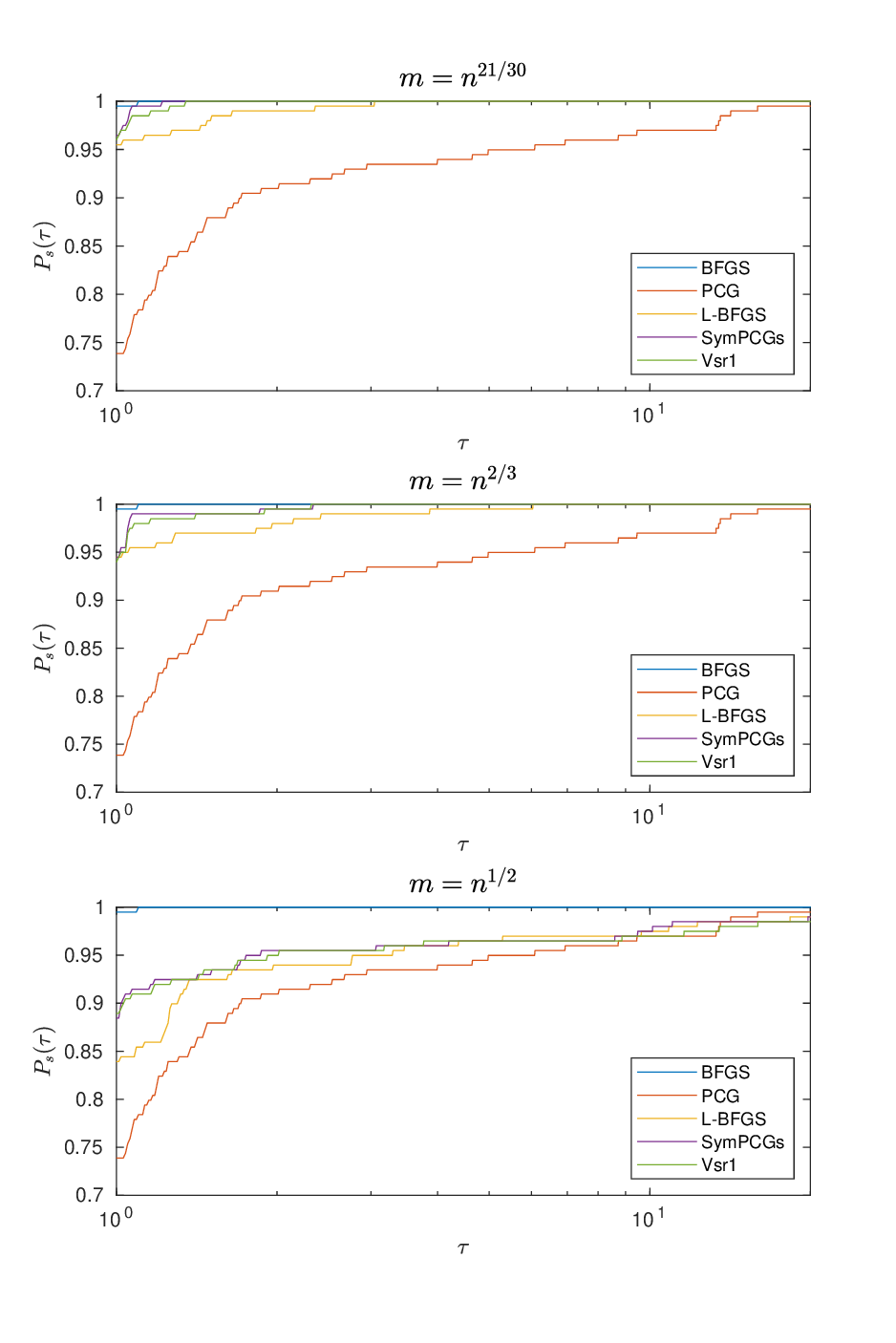}}
%\qquad
\begin{minipage}{5.7cm}
\includegraphics[width=6.5cm]{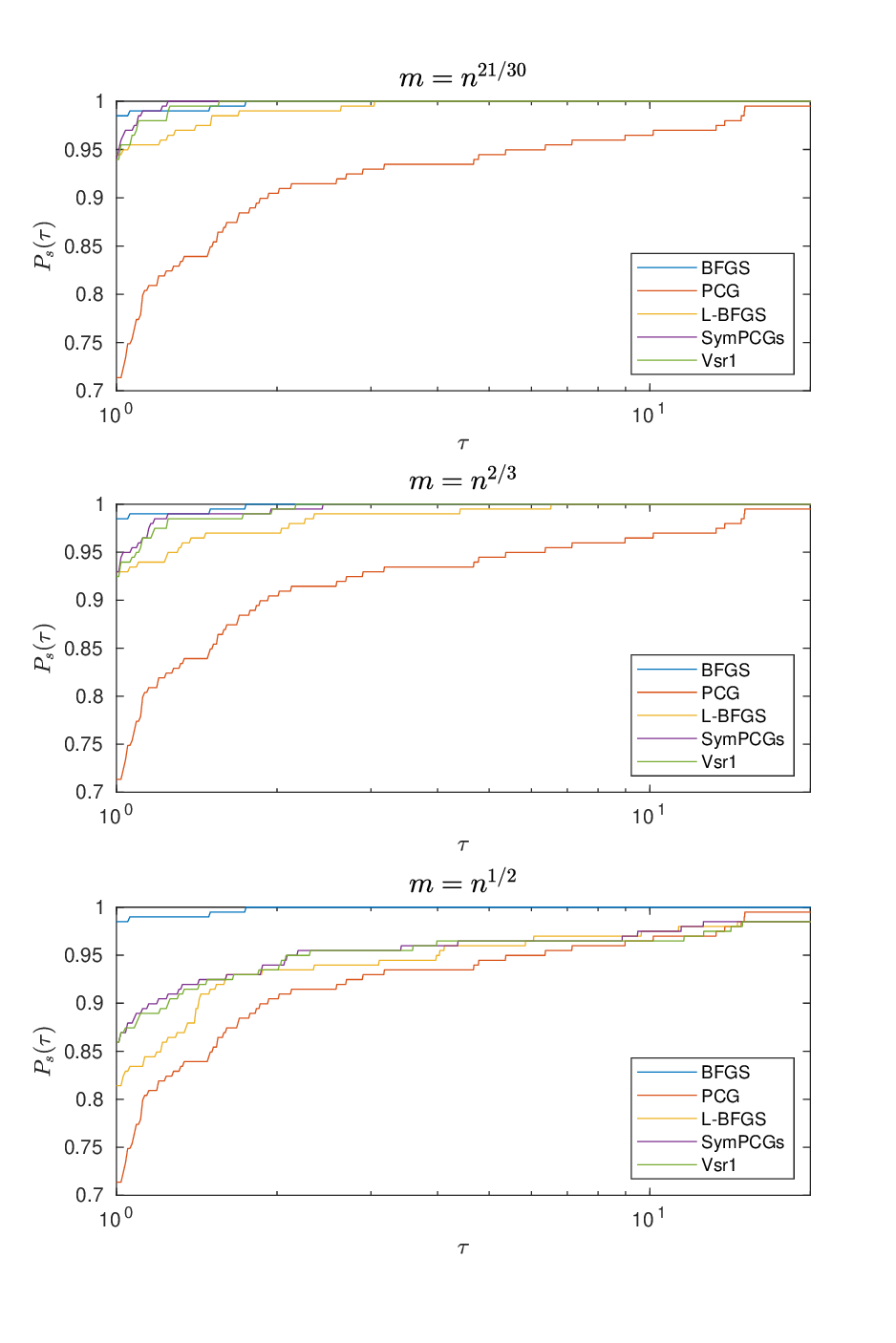} 
\end{minipage}
\caption{Performance profiles with $B_0 = \nabla^2 f(x_{j-1})$. The left figure corresponds to stopping criteria $\| g_k \| < 10^{-7}$ and the right to $\| g_k \| \leq \max \{\epsilon^{DS}, 10^{-13}\}$.}
\label{fig:PPwPC}
\end{figure}
\vspace{-2mm}
\begin{table}[H]
%\captionsetup{font=footnotesize,labelsep=colon,labelfont=bf} 
\caption{Average number of iterations required per system of linear equations.}
\begin{footnotesize}
\begin{tabular}{ccccccc}
                                       &                               &                           &                          & {[}$n^{\frac{1}{2}}, n^{\frac{2}{3}}, n^{\frac{21}{30}}${]}       & {[}$n^{\frac{1}{2}}, n^{\frac{2}{3}}, n^{\frac{21}{30}}${]}         & {[}$n^{\frac{1}{2}}, n^{\frac{2}{3}}, n^{\frac{21}{30}}${]}  \\
$B_0$                                  & \multicolumn{1}{|c|}{Accuracy} & \multicolumn{1}{c|}{BFGS} & \multicolumn{1}{c|}{PCG} & \multicolumn{1}{c|}{L-BFGS} & \multicolumn{1}{c|}{\texttt{SymPCGs}} & \texttt{Vsr1}                  \\ \hline
\multirow{2}{*}{$I$}                   & \multicolumn{1}{|c|}{low}      & 223                   &                    \multicolumn{1}{|c|}{845} &   [842, 753, 704]      &   \multicolumn{1}{|c|}{[860, 616, 589]}             &     [858, 618, 581]          \\
                                       & \multicolumn{1}{|c|}{high}     & 319                       &                    \multicolumn{1}{|c|}{994} &   [997, 930, 912]      &   \multicolumn{1}{|c|}{[1020, 819, 778]}             &     [1022, 815, 770]          \\
\multirow{2}{*}{$\nabla^2 f_{-}$} & \multicolumn{1}{|c|}{low}      & 24                       &                    \multicolumn{1}{|c|}{125}  &   [127, 37, 30]       &   \multicolumn{1}{|c|}{[118, 26, 24]}    &     [125, 27, 24]          \\
                                       & \multicolumn{1}{|c|}{high}     & 29                       &     \multicolumn{1}{|c|}{141} &   [135, 45, 35]       &   \multicolumn{1}{|c|}{[132, 32, 29]}     &     [140, 32, 30]     
\end{tabular} \label{table:aveNumIt}
      \end{footnotesize}
\end{table}

\begin{table}[H]
%\captionsetup{font=footnotesize,labelsep=colon,labelfont=bf} 
\centering
\caption{Average size relative to $m$ of the systems solved per iteration.}
\begin{footnotesize}
\begin{tabular}{cccccccc}
                                       &                               & \multicolumn{3}{c}{\texttt{SymPCGs}} & \multicolumn{3}{c}{\texttt{Vsr1}} \\
$B_0$                                  & \multicolumn{1}{|c|}{Accuarcy} & $n^{\frac{1}{2}}$    & $n^{\frac{2}{3}}$   & \multicolumn{1}{c|}{$n^{\frac{21}{30}}$}   & $n^{\frac{1}{2}}$    & $n^{\frac{2}{3}}$   & $n^{\frac{21}{30}}$   \\ \cline{1-8}
\multirow{2}{*}{$I$}                   & \multicolumn{1}{|c|}{low}      &   0.984 & 0.939  & \multicolumn{1}{c|}{0.924}  &   0.984 &     0.949 &  0.923    \\
                                       & \multicolumn{1}{|c|}{high}     & 0.985  &  0.948 &\multicolumn{1}{c|}{0.933}  &  0.985  &      0.948 &    0.933  \\
\multirow{2}{*}{$\nabla^2 f_-$} & \multicolumn{1}{|c|}{low}      &  0.929 &  0.443 &  \multicolumn{1}{c|}{0.335}  &   0.933&      0.446 &   0.342   \\
                                       & \multicolumn{1}{|c|}{high}     & 0.927 &  0.457  & \multicolumn{1}{c|}{0.345}  &   0.931 &      0.461 &     0.355
\end{tabular}\label{table:aveSizeSolv}
      \end{footnotesize}
\end{table}

\begin{figure}[H]
%\centering
%\captionsetup{font=footnotesize,labelsep=colon,labelfont=bf} 
\parbox{5.9cm}{
\includegraphics[width=6.5cm]{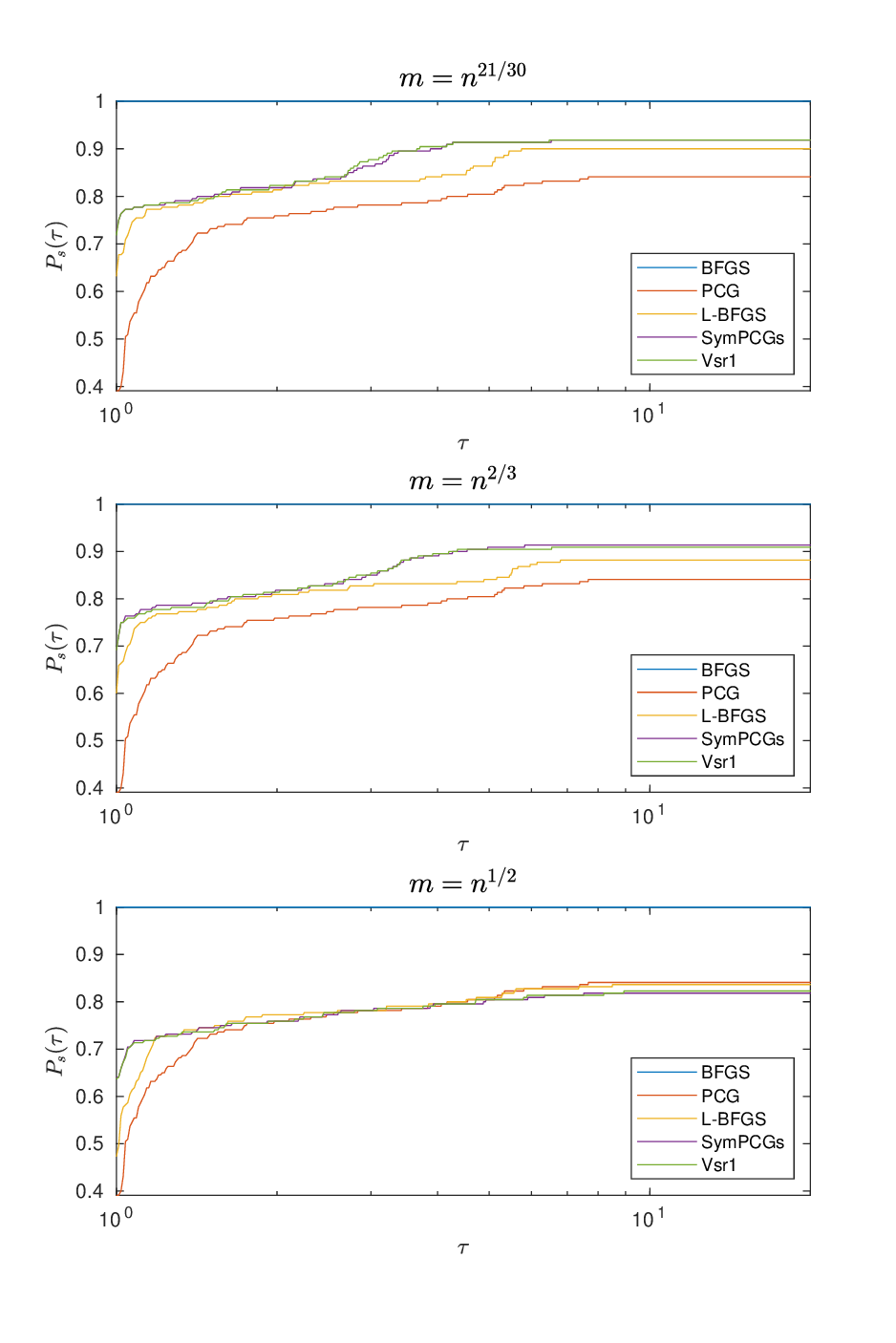}}
%\qquad
\begin{minipage}{5.7cm}
\includegraphics[width=6.5cm]{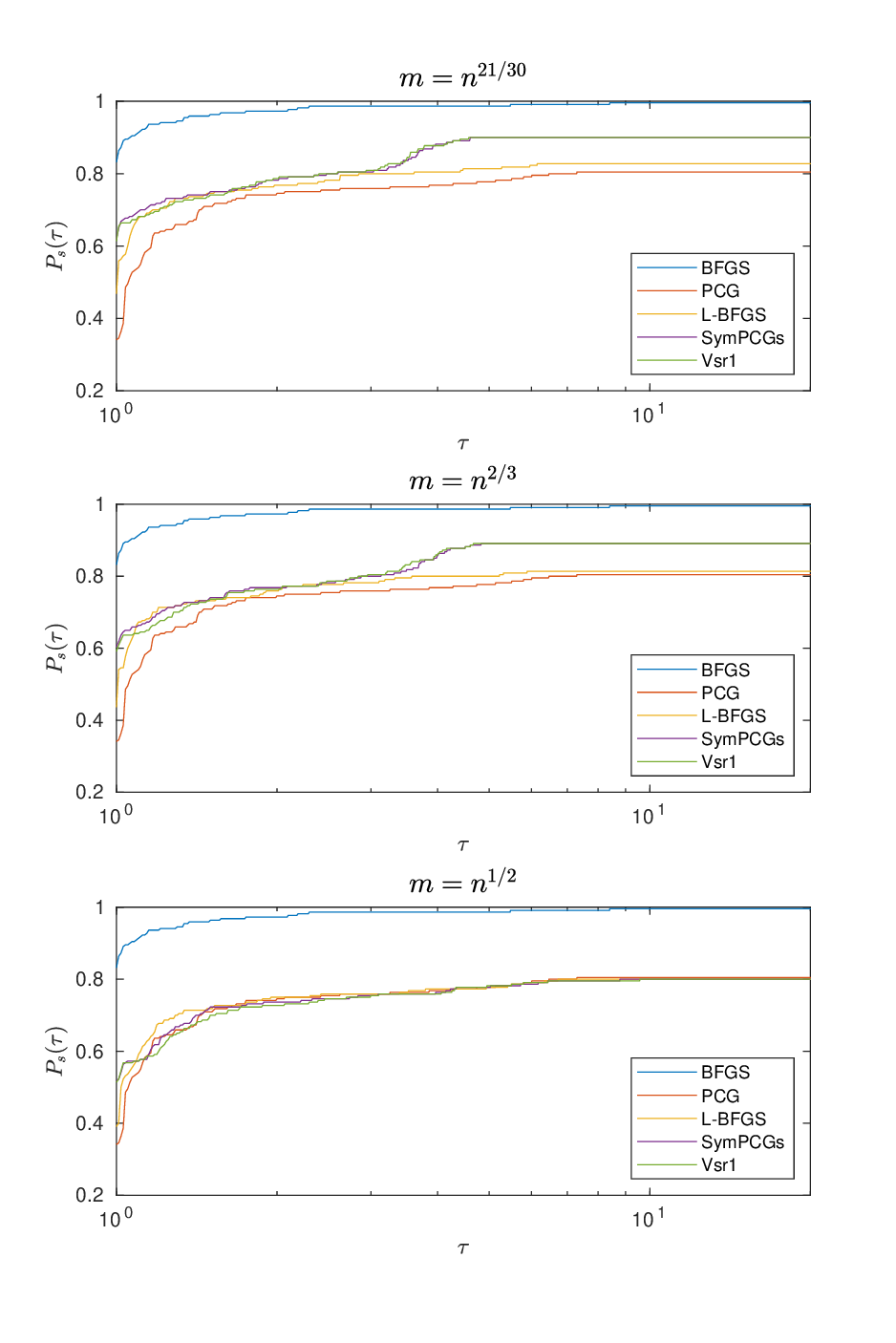} 
\end{minipage}
\caption{Performance profiles with $B_0 = I$. The left figure corresponds to stopping criteria $\| g_k \| < 10^{-7}$ and the right to $\| g_k \| \leq \max \{\epsilon^{DS}, 10^{-13}\}$.}
\label{fig:PPwoPC}
\end{figure}

The results in Figure~\ref{fig:PPwPC}-\ref{fig:PPwoPC} and
Table~\ref{table:aveNumIt} show that the BFGS method has the best
overall performance and PCG the worst. With preconditioner
\texttt{SymPCGs} and \texttt{Vsr1} outperform L-BFGS for the two
larger limited memory values $m$. This difference is not as distinct
in the non-preconditioned case, as can be seen in
Figure~\ref{fig:PPwoPC} even though \texttt{SymPCGs} and \texttt{Vsr1}
are able to solve more problems within the maximum number of
iterations compared to L-BFGS. Numerical experiments have shown that,
depending on the particular instance, small $m$-values can lead to
loss of convergence rate close the solution. Tendencies to this can be
seen in Figure~\ref{fig:PPwPC}-\ref{fig:PPwoPC} and
Table~\ref{table:aveNumIt} for the limited-memory methods with
$m=n^\frac{1}{2}$, which is around $30$ for problems of size
$10^3$. Further decreasing $m$ also increases this tendency. The
memory parameter $m$ might be considered to be large however, as
discussed in Appendix in regards to complexity, this value does not
have a significant impact on the complexity when $m \leq
n^\frac{2}{3}$. As Table~\ref{table:aveSizeSolv} also shows, most
systems solved in the preconditioned case is of size less than
$m$. The slight difference between the low and high tolerances in
Table~\ref{table:aveNumIt}-\ref{table:aveSizeSolv} indicates that the
bulk of the work is made prior to reaching low tolerance level.
Numerical experiments has further shown that the behavior on a
specific system of linear equations not only depends on the choices
when designing the method but also on the individual properties of the
linear system. We choose to give the results for the methods denoted
by \texttt{SymPCGs} and \texttt{Vsr1} since we believe that these are
representative for the proposed limited-memory quasi-Newton class,
given the complexity restrictions, on the test set. However we have
observed that the results can be improved and dis-improved for
different choices of the quantities in Table~\ref{table:MethodPar}. We
conclude that there is a potential for accelerating the process of
solving sequences of related systems of linear equations with
iterative methods in the proposed framework.

\section{Conclusion} \label{section:conc}

In this work we have given a class of limited-memory quasi-Newton
Hessian approximations which on (\ref{eq:QP}) with the exact
linesearch method of Algorithm~\ref{alg:ELSMQP} generate $p_k$
parallel to $p_k^{PCG}$. With the framework of reduced-Hessians this
provides a dynamical framework for the construction of limited-memory
quasi-Newton methods.  In addition, we have characterized all
symmetric rank-two update matrices, $U_k$ with $\mathcal{R}(U_k) =
\textit{span}\big(\{ g_{k-1}, g_k \}\big)$ which gives $p_k$ parallel
to $p_k^{PCG}$ in this setting. The Hessian approximations were
described by a novel compact representation whose framework was first
presented in Section~\ref{section:CR} for the full Broyden class on
unconstrained optimization problems (\ref{eq:uncP}). The
representation of the full Broyden class consists only of explicit
matrices and gradients as vector components.

Numerical simulations on randomly generated unconstrained quadratic
optimization problems have shown that for these problems our suggested
multi-\linebreak parameter class, with parameters within a certain range, is
equivalent to the BFGS method in finite arithmetic. Simulations on
sequences of related systems of linear equations which originate from
the CUTEst test collection have given an indication of the potential
of the the proposed class of limited-memory methods. The results
indicate that there is a potential for accelerating the process of
solving sequences of related systems of linear equations with
iterative methods in the proposed framework.

The results of this work are meant to contribute to the theoretical
and numerical understanding of limited-memory quasi-Newton methods for
minimizing a quadratic function. We envisage that they can lead to
further research on limited-memory methods for unconstrained
optimization problems. In particular, limited-memory methods for
minimizing a near-quadratic function and for systems arising as
interior-point methods converge.

\subsubsection*{Acknowledgements}
  We thank Elias Jarlebring for valuable discussions on finite
  precision arithmetic. In addition, we thank the referees and
  the associate editor for many helpful comments and suggestions which
  significantly improved the presentation.
\newpage
\begin{footnotesize}
\section*{Appendix}
\setcounter{section}{1}
\setcounter{table}{0}
\setcounter{equation}{0}
\counterwithin{table}{section}
\counterwithin{equation}{section}
\renewcommand{\thesection}{\Alph{section}}%
The following lemma relates a symmetric rank-one update in the matrix
to a scaling of the solution when the rank-one vector is a multiple of
the right-hand side.

\begin{lemma}\label{Ap:lemma:Axb}
If $Ax=b$, with $A$ nonsingular then 
\begin{equation*}
\left( A+\gamma b b^T \right)y = b, \ \text{for} \ y= \frac1{1+\gamma b^Tx} x,
\end{equation*}
if $1+\gamma b^Tx \neq 0$. If, in addition, $b\T x\ne 0$, it holds that
\begin{equation}
\frac1{b\T y}=\frac1{b\T x}+\gamma.
\label{Ap:eq:lemma1eq1}
\end{equation}
Finally, if $A=A^T\succ 0$, then $b\T x >0$ and $A+\gamma b b^T \succ
0$ if and only if
\begin{equation*}
\gamma > -\frac1{b\T x}.
\end{equation*}
\end{lemma}
\begin{proof}
Assume that $Ax=b$ where $A$ is nonsingular. Premultiplication of $\left( A+\gamma b b^T \right)y  = b$ by $A\inv$ gives
\begin{equation}
\left( I + \gamma A\inv b b^T \right) y = A\inv b.
\label{Ap:eq:l1e1}
\end{equation}
Insertion of $x = A\inv b$ into (\ref{Ap:eq:l1e1}) and rearranging gives
 \begin{equation}
y = \left( 1- \gamma b^Ty \right)x.
\label{Ap:eq:l1e2}
\end{equation} 
Insertion of $y = \alpha x$ into (\ref{Ap:eq:l1e2}) and solving for $\alpha$ yields  
\begin{equation*}
\alpha = \frac1{1+\gamma b^Tx}, \qquad 1+\gamma b^Tx \neq 0.
\end{equation*}
The result in (\ref{Ap:eq:lemma1eq1}) follows by premultiplication of $y= \frac1{1+\gamma b^Tx} x$ by $b^T$ and rearranging. For the final result, note that $b\T x = x^T A x > 0$ since $A \succ 0$ and that
\begin{equation*}
\left( A+\gamma b b^T \right) = A^{1/2}\left( I+\gamma A^{-1/2} b b^T A^{-1/2} \right) A^{1/2},
\end{equation*}
which is a congruent transformation and hence $I+\gamma A^{-1/2} b b^T A^{-1/2} \succ 0$ if and only if $A+\gamma b b^T \succ 0$. Then consider the similarity transformation
\begin{equation*}
A^{-1/2}\left( I+\gamma A^{-1/2} b b^T A^{-1/2} \right) A^{1/2} = I + \gamma xb^T,
\end{equation*}
where the only eigenvalue not equal to unity is $1 + \gamma b\T x$, which is positive only if $\gamma > - \frac1{b\T x}$, $b\T x \neq 0$. 
\end{proof}

Next, we give a nonsingularity condition on a class of
quasi-Newton matrices that we consider.

\begin{lemma} \label{Ap:lemma:Bposdef}
  Let $B_0 \in \mathcal{S}_+^n$ and let $g_i$,
  $i=0,\dots,k$, be nonzero vectors that are conjugate with respect to
  $B_0\inv$. Define $B_k$ as
\begin{equation*}%\label{Ap:eq:Bmat}
B_k = B_0 + \sum_{i=0}^{k-1} \left( -\frac1{g_i\T B_0\inv g_i} g_i g_i^T + 
  \rho_i (g_{i+1}-g_i)(g_{i+1}-g_i)^T \right),
\end{equation*}
where $\rho_i \in \mathbb{R}$, $i=0,\dots,k-1$. Then $B_k \succ 0$ if
$\rho_i > 0$, $i=1,\dots,k-1$.
\end{lemma}

\begin{proof} Any vector $p$ in $\mathbb{R}^n$ can be written as
\begin{equation}\label{eqn-p}
p=\sum_{i=0}^{k} \alpha_i B_0\inv g_i + B_0 \inv u, \text{\ with\ } g_i^T
B_0\inv u=0,  \quad i=0,\dots,k.
\end{equation}
Insertion of (\ref{eqn-p}) into $p\T B_k p$ gives
\begin{align*}%\label{eqn-pBp}
  p\T B_k p & =  p\T \left(B_0 + \sum_{i=0}^{k-1} \left[ -\frac1{g_i\T B_0\inv g_i} g_i g_i^T +
      \rho_i (g_{i+1}-g_i)(g_{i+1}-g_i)^T \right]\right) p \nonumber \\
   & =  p\T B_0 p - \sum_{i=0}^{k-1} \frac{\left( g_i^T p \right)^2}{g_i\T B_0\inv g_i}
+ \sum_{i=0}^{k-1} \rho_i \left( (g_{i+1}-g_i)^T p\right)^2 \nonumber \\
& =  \sum_{i=0}^{k} \alpha_i^2 g_i\T B_0\inv g_i +u\T B_0\inv u- \sum_{i=0}^{k-1} \frac{\left( \alpha_i g_i\T B_0\inv g_i \right)^2}{g_i\T B_0\inv g_i} + \sum_{i=0}^{k-1} \rho_i \left((g_{i+1}-g_i)^T p\right)^2 \nonumber \\
& =  \alpha_k^2 g_k\T B_0\inv g_k + u\T B_0\inv u  + \sum_{i=0}^{k-1}
\rho_i \left(\alpha_{i+1} g_{i+1} ^T B_0\inv g_{i+1} -
\alpha_{i} g_{i} ^T
  B_0\inv g_{i}\right)^2.
\end{align*}
For the remainder of the proof, let $\rho_i > 0$, $i = 0,\dots,
k-1$. Then $B_k$ is positive
semidefinite with $p\T B_k p=0$ only if
\begin{subequations}
\begin{align}
\alpha_k^2 g_k\T B_0\inv g_k & = 0, \label{eqn-zeroa} \\
u\T B_0\inv u &= 0, \label{eqn-zerob} \\
\alpha_{i+1} g_{i+1}^T B_0\inv g_{i+1} -\alpha_{i} g_{i}^T
  B_0\inv g_{i} & = 0, \quad i=0,\dots,k-1. \label{eqn-zeroc}
\end{align}
\end{subequations}
From the positive definiteness of $B_0$, (\ref{eqn-zeroa}) gives
$\alpha_k=0$, which in combination with  (\ref{eqn-zeroc}) gives
$\alpha_i=0$, $i=0,\dots,k$. In addition, (\ref{eqn-zerob}) gives
$u=0$. Therefore, $p\T B_k p=0$ only if $p=0$, proving that $B_k$ is
positive definite. 
\end{proof}
In Theorem~\ref{thm:LMQN}, it has been shown how search directions
parallel to PCG may be generated in a limited-memory setting if
search directions are included in the basis in addition to
gradients. The following theorem gives a way of using gradients only
by modifying the right-hand side.
\begin{theorem}\label{thm:LMCgrad}
  Consider iteration $k$, $1\leq k < r$, of the exact-linesearch method of
  Algorithm~$\ref{alg:ELSMQP}$ for solving $(\ref{eq:QP})$. Assume that
  $p_i=\delta_i p_i^{PCG}$ with $\delta_i \neq 0$ for $i= 0, \dots ,
  k-1$, where $p_i^{PCG}$ is the search direction
  of PCG with associated $B_0 \in \mathcal{S}_+^n$, as stated in
  $(\ref{eq:pPCG})$. Let $\mathcal{A}_k = \{ j_1, \dots , j_{m_k}\} \subseteq \{ 0, \dots, k \}$ with $j_1<j_2< \dots < j_{m_k}$ such that $j_{m_k} = k$ and let $\mathcal{I}_k = \{ 0, \dots
  k\} \setminus \mathcal{A}_k$. Furthermore, let $p_k$ satisfy $B_kp_k = -N_k g_k$ where
\begin{subequations}
\begin{align} 
B_k & = B_0 +\sum_{i=1}^{m_k-1} \left(-\frac1{g_{j_i}^TB_0\inv g_{j_i}}g_{j_i}
  g_{j_i}^T +
  \rho_{j_i}^{(k)} (g_{j_{i+1}}-g_{j_i})(g_{j_{i+1}}-g_{j_i})^T\right)
 \nonumber \\ & \quad + \left( \frac{1}{\delta_k}-1 \right)\frac{1}{g_k^TB_0^{-1}g_k} g_{k} g_{k}^T, \label{eq:BgenLC}
 \end{align}
and
 \begin{align} 
N_k  =  I + \delta_k \sum_{i \in \mathcal{I}_k} \left(
  \frac1{g_i^T B_0 \inv g_i }\right) g_i g_k^T B_0^{-1},
\label{eq:hgenLC}
\end{align}
\label{eq:genLC}
\end{subequations}
\hspace{-4pt} where $\delta_k$ and $\rho_{j_i}^{(k)}$, $i = 1, \dots, m_k-1 $, are chosen such that $B_k$ is nonsingular. Then \[ p_k = \delta_k p_k^{PCG}.\] In particular, if $\rho_{j_i}^{(k)}>0$, $i = 1, \dots, m_k-1 $, and
$\delta_k > 0$, then $B_k\succ0$.
\end{theorem}
\begin{proof}
We will show that $p_k = \delta_k p_k^{PCG}$, $\delta_k \neq 0$ satisfies $B_k p_k = -N_k g_k$. Note that by Lemma~\ref{lemma:conjProp} it follows that $g_i^T p_k^{PCG} = - g_k^T B_0 \inv g_k$, $i=0,\dots, k$ and hence
\begin{equation} \label{eq:appendixProofEq1}
 \delta_k B_k p_k^{PCG} = \delta_k \left( B_0 p_k^{PCG} +  \sum_{i=1}^{m_k-1} \frac{g_k^T B_0^{-1} g_k}{g_{j_i}^TB_0\inv g_{j_i}}g_{j_i} 
   - \left( \frac{1}{\delta_k}-1 \right)g_{k} \right).
\end{equation}
Insertion of $p_k^{PCG}$ as in (\ref{eq:pPCGexp}) with $M = B_0$ into (\ref{eq:appendixProofEq1}) gives
\begin{align*}
 \delta_k B_k p_k^{PCG} & = \delta_k \left(- \sum_{i=0}^k \frac{g_k^T B_0^{-1} g_k}{g_i^TB_0^{-1}g_i} g_i +   \sum_{i=1}^{m_k-1} \frac{g_k^T B_0^{-1} g_k}{g_{j_i}^TB_0\inv g_{j_i}}g_{j_i} 
   - \left( \frac{1}{\delta_k}-1 \right)g_{k} \right)  \\
   & = -\delta_k \sum_{i \in \mathcal{I}_k}
   \frac{g_k^T B_0^{-1}g_k}{g_i^T B_0 \inv g_i } g_i  - g_k =
-N_k g_k,
\end{align*}
with $N_k$ given by (\ref{eq:hgenLC}). Thus $p_k = \delta_k p_k^{PCG}$,
$\delta_k \neq 0$ is a solution to $B_k p_k = -N_k g_k$, since $B_k$
is nonsingular this is also the unique solution. The matrix $B_k$ is
positive definite with $\delta_k = 1$ and $\rho_{j_i}^{(k)} > 0$,
$i=1,\dots, m_k-1$ by Lemma~\ref{Ap:lemma:Bposdef}. The steps to show
that $B_k \succ 0$ if it in addition holds that $\delta_k>0$ is
analogues to the last steps of Lemma~\ref{Ap:lemma:Axb}. 
\end{proof}
If all indices are chosen to be active in the Hessian approximation of (\ref{eq:BgenLC}) then it is equivalent to (\ref{eq:BMuP}) of Proposition~\ref{prop:genHuangQP} where $\varphi_k$ relates to $\delta_k$ as in (\ref{eq:1to1}). Conversely, if the indices corresponding to all previous gradients, i.e. $i = 0, \dots, k-1$ are inactive, and $\delta_k = 1$ then (\ref{eq:genLC}) with \linebreak $B_k p_k = -N_k g_k$ is equivalent to the PCG update (\ref{eq:pPCGexp}). The update scheme of Theorem~\ref{thm:LMCgrad} contains only gradients as vector components and with the exact linesearch method of Algorithm~\ref{alg:ELSMQP} the finite termination property is maintained. However, this is at the expense of adding a correction term on the right hand side. Moreover, the update in Theorem~\ref{thm:LMCgrad} relies heavily on the result in Lemma~\ref{lemma:conjProp} which is exact on quadratic problems. For non-quadratic problems other more accurate approximations and modifications may be considered to improve the method. 

\subsection*{Complexity of \texttt{symPCGs} and \texttt{Vsr1}}  \label{sec:complexity}
It is sufficient to consider the complexity for iteration $k \geq
m$. Note that the first $(m-3)$ columns of $Z_k$ remain constant since
the first $(m-3)$ columns of $Q_k$ remain the same. The basis matrix
may hence be written as $Z_k = \begin{pmatrix} Z_0 &
  \bar{Z}_{k} \end{pmatrix}$ where $Z_0 \in \mathbb{R}^{n \times
  (m-3)}$ and $\bar{Z}_k \in \mathbb{R}^{n \times 3}$. This reduces
the complexity of the Gram-Schmidt process to $O(mn)$. Moreover, for
both \texttt{SymPCGs} and \texttt{Vsr1} the matrix $B_k$ can be written as \[ B_k = B_0 + V_0^\rho + F_k , \]
where $V_0^\rho = \sum_{i=0}^{m-4} \rho_i^B (g_{i+1}-g_i)(g_{i+1}-g_i)^T$ and $F_k = V_k + \sum_{i=k-3}^{k-1} \rho_i^B (g_{i+1}-g_i)(g_{i+1}-g_i)^T$. Note that $F_k$ has a compact representation of rank-five in \texttt{symPCGs} and of rank-four in \texttt{Vsr1} respectively. Moreover, note that the matrix multiplication $\hat{M} F \hat{N}$ where $\hat{M} \in \mathbb{R}^{q_1 \times n}$, $\hat{N} \in \mathbb{R}^{n \times q_2}$ and \linebreak $F \in \mathbb{R}^{n \times n}$ is of rank $\hat{r}$ with a known compact representation, has complexity \linebreak $O(nq_1\hat{r} +  q_1  q_2\hat{r} + n  q_2\hat{r})$. The reduced-Hessian can at iteration $k$ be written as
\begin{equation*}
Z_k^T B_k Z_k = \begin{pmatrix}
Z_0^T ( B_0 + V_0^\rho ) Z_0 + Z_0^T F_k Z_0 & Z_0^T (  B_0 + V_0^\rho )  \bar{Z}_k  + Z_0^T  F_k  \bar{Z}_k \\
\bar{Z}_k^T (  B_0 + V_0^\rho ) Z_0 + \bar{Z}_k^T F_k Z_0 & \bar{Z}_k^T (  B_0 + V_0^\rho + F_k ) \bar{Z}_k \end{pmatrix}.
\end{equation*}
The matrices $Z_0^T ( B_0 + V_0^\rho ) Z_0 \in \mathbb{R}^{(m-3)\times (m-3)}$ and $(B_0 + V_0^\rho)Z_0 \in \mathbb{R}^{n\times (m-3)}$ have been successively built for iteration $k < m-3$ and can be stored. What remains to be computed and the corresponding complexity, taking into account that $\hat{r}\leq 5$, is shown in Table~\ref{table:complexity}.
  \begin{table}[H]
  \footnotesize
      \centering
%\captionsetup{font=footnotesize,labelsep=colon,labelfont=bf} 
      \caption{Complexity for computing quantities in the reduced-Hessian, $Z_k^T B_k Z_k$.}
\begin{tabular}{c|c|c|c} \label{table:complexity}
  $Z_0^T F_k Z_0$ & $\bar{Z}_k^T (  B_0 + V_0^\rho ) Z_0$ & $ \bar{Z}_k^T F_k Z_0$ & $\bar{Z}_k^T (  B_0 + V_0^\rho + F_k ) \bar{Z}_k$ \\ \hline\noalign{\smallskip} 
 $O( m^2 + nm)$ &  $O(nm) $ & $ O(n+m+nm)$ & $O(n+n^2)$
      \end{tabular}  
\end{table} 
\noindent
It remains to add everything together, which has complexity $O(m^2)$, and to factorize, which has complexity $m^3$. The overall asymptotic complexity is thus dominated by $\max\{n^2, m^3\}.$\\
\end{footnotesize}

\newpage
\bibliographystyle{myplain}  
\bibliography{references,references2} % name your BibTeX data base

\end{document}